\documentclass[12pt,a4paper,oneside]{amsart}
\usepackage[utf8]{inputenc}
\usepackage{amsmath,amscd,hyperref, amsfonts, amssymb, amsthm}
\usepackage{setspace,kantlipsum,enumerate}
\usepackage{tikz-cd}
\usepackage{mathrsfs}
\usepackage{textcomp}
\usepackage{colonequals}
\usepackage[margin=1in]{geometry}
\usepackage{mathtools}

\newtheorem*{assumption}{Assumption}

\newcommand{\sing}{\mathrm{sing}}
\newcommand{\bl}{\operatorname{bl}}

\newcommand{\Jac}{\operatorname{Jac}}
\newcommand{\Pic}{\mathrm{Pic}}
\newcommand{\Sym}{\mathrm{Sym}}
\newcommand{\Sec}{\operatorname{Sec}}

\newcommand{\pr}{\mathrm{pr}}

\newcommand{\restr}[1]{\big\vert_{#1}}

\DeclareMathOperator{\Ker}{Ker}
\DeclareMathOperator{\id}{id}
\theoremstyle{plain} 
\newtheorem{thm}{Theorem}[section]
\newtheorem{lemma}[thm]{Lemma}
\newtheorem{prop}[thm]{Proposition}
\newtheorem{corollary}[thm]{Corollary}
\newtheorem{claim}[thm]{Claim}

\theoremstyle{definition}
\newtheorem{notation}[thm]{Notation}

\newtheorem{definition}[thm]{Definition}

\newtheorem{remark}[thm]{Remark}
\newtheorem{problem}[thm]{Problem}

\newtheorem{fact}[thm]{Fact}

\newtheorem{question}[thm]{Question}
\newcommand{\C}{\mathbb{C}}

\newcommand{\cD}{\mathcal{D}}

\newcommand{\cI}{\mathcal{I}}

\newcommand{\cL}{\mathcal{L}}

\newcommand{\cO}{\mathcal{O}}

\newcommand{\bP}{\mathbf{P}}

\newcommand{\fm}{\mathfrak{m}}

\newcommand{\sD}{\mathscr{D}}
\newcommand{\sE}{\mathscr{E}}

\newcommand{\ctbl}[1]{T^{\ast}_{#1}}

\newcounter{tmp}
\begin{document}

\title[Log resolution for hyperelliptic theta divisors]
{A log resolution for the theta divisor of a hyperelliptic curve}

\author{Christian Schnell and Ruijie Yang}
\begin{abstract}

In this paper, we prove that the theta divisor of a smooth hyperelliptic curve has a
natural and explicit embedded resolution of singularities using iterated blowups of
Brill-Noether subvarieties. We also show that the Brill-Noether stratification of the hyperelliptic Jacobian is a Whitney stratification.
\end{abstract}
\maketitle
\setcounter{tocdepth}{1}

\section*{Introduction}

Let $C$ be a smooth projective curve of genus $g \geq 1$. Let $\Jac(C)$ be the
Jacobian of $C$, and let $\Theta \subseteq \Jac(C)$ be the theta divisor. 
The purpose of this paper is to give a natural and explicit log resolution of the
pair $(\Jac(C),\Theta)$ when $C$ is a hyperelliptic curve.

Recall that the Brill-Noether variety $W^r_{g-1}(C)$ parametrizes line bundles $L\in
\Pic^{g-1}(C)$ of degree $g-1$ with $h^0(L)\geq r+1$. According to a theorem by
Riemann, we can choose an isomorphism $\Jac(C)\cong \Pic^{g-1}(C)$ so that the theta
divisor $\Theta$ becomes identified with $W_{g-1}(C)\colonequals W^0_{g-1}(C)$. The
Abel-Jacobi map from the symmetric product $C_{g-1} \colonequals \Sym^{g-1}(C)$ gives a
resolution of singularities of
$\Theta$, which is useful for answering many geometric questions about Jacobian varieties.
However, if one wants to investigate the geometry of the embedding $\Theta\subseteq
\Jac(C)$, one needs instead a log resolution of the pair $(\Jac(C),\Theta)$.
Inspired by a global study of the vanishing cycle functor for divisors \cite{SY23}, we are lead to the question of finding an explicit log resolution in the
case of hyperelliptic theta divisors. Since the log resolution is of a purely geometric
nature, we leave the actual computation of vanishing cycles to another paper \cite{SY23}.

When $C$ is a hyperelliptic curve of genus $g \geq 1$, we have a lot of very precise
information about the chain of subvarieties
\begin{align}\label{eqn: Brill-Noether sequence}
\Theta=W_{g-1}(C) \supseteq W^1_{g-1}(C) \supseteq \cdots \supseteq W^{n}_{g-1}(C),
\end{align}
where $n = \bigl\lfloor \frac{g-1}{2} \bigr\rfloor$ is the maximal integer such that
$W^n_{g-1}(C)\neq \varnothing$. First, the dimension of $W^r_{g-1}(C)$ is equal to
$g-1-2r$ and $W^r_{g-1}(C)$ is reduced (see Proposition \ref{prop: reducedness of Wrd}). Second, the singular locus of $W^r_{g-1}(C)$ is exactly
$W^{r+1}_{g-1}(C)$. Third, the multiplicity of the theta divisor at a
point $L \in \Pic^{g-1}(C)$ is equal to $h^0(L)$ by the Riemann singularity theorem,
and so $W^r_{g-1}(C)$ is exactly the set of points of multiplicity $\geq r+1$ on $\Theta$ (see \cite[Chapter IV, \S 4]{ACGH} for details).

These facts immediately suggest that one might be able to get a log resolution of
the pair $(\Jac(C), \Theta)$ by successively blowing up the Brill-Noether subvarieties
$W_{g-1}^r(C)$ in the order from smallest to largest. This guess turns out to be correct, but it requires
quite a bit of work to prove rigorously that it works. 

More precisely, we use the following iterative procedure, consisting of $n$ steps. In
the first step, we blow up $\Jac(C)$ along the smallest subvariety $W^n_{g-1}(C)$,
and denote the blowup by $\pi_1 \colon \bl_1(\Jac(C)) \to \Jac(C)$. In the second
step, we blow up $\bl_1(\Jac(C))$ along the strict transform of $W^{n-1}_{g-1}(C)$,
and denote the new blowup by $\pi_2 \colon
\bl_2(\Jac(C)) \to \Jac(C)$. In the $i$-th
step, we blow up $\bl_{i-1}(\Jac(C))$ along the strict transform of
$W^{n+1-i}_{g-1}(C)$, and denote the new blowup by $\pi_i \colon \bl_i(\Jac(C)) \to
\Jac(C)$. This process stops after the $n$-th step. The strict transforms of the exceptional divisor give us a sequence
of divisors $Z_0, Z_1, \dotsc, Z_{n-1}$, with $Z_i$ sitting over the locus
$W_{g-1}^{n-i}(C)$. Let $\tilde{\Theta}$ denote the
strict transform of the theta divisor. With this notation, our main result is the following.

\begingroup
\setcounter{tmp}{\value{thm}}
\setcounter{thm}{0} 
\renewcommand\thethm{\Alph{thm}}
\begin{thm}\label{theorem: log resolution of hyperelliptic theta divisors}
The morphism $\pi_n:\bl_{n}(\Jac(C))\to \Jac(C)$ 
is a log resolution of $(\Jac(C),\Theta)$, where
\[\pi_n^{\ast}(\Theta)=\tilde{\Theta}+\sum_{i=0}^{n-1} (n+1-i) Z_i\] 
is a divisor with simple normal crossing support. Moreover, the strict transform of $W^{n-i}_{g-1}(C)$ in $\bl_i(\Jac(C))$ is smooth. In other words, each
blowup in the sequence is a blowup along a smooth center.  
\end{thm}

We can also describe the generic structure of the exceptional divisors.

\begin{corollary}\label{corollary: generic structure of exceptional divisors}
	For $r = 1, \dotsc, n$, every fiber of the projection
	\[
		Z_{n-r} \setminus (\tilde{\Theta}\cup \bigcup_{\substack{0\leq i\leq n-1 \\ i\neq n-r}}Z_i)
		\to W_{g-1}^{r}(C) \setminus W_{g-1}^{r+1}(C)
	\]
	is isomorphic to the complement of a hypersurface of degree $r+1$ in $\bP^{2r}$;
	the hypersurface is the $(r-1)^{
	\mathrm{th}}$ secant variety of a rational normal curve of degree $2r$ in $\bP^{2r}$.
\end{corollary}

There are a few other examples in the literature where this simple-minded procedure
of successive blowups along singular loci produces a log resolution.
\begin{enumerate}
	 \item Let $X$ be the affine space of $n$-by-$n$ matrices and let $D$ be the hypersurface
		 defined by the vanishing of the determinant function. Let $D_i \subseteq D$ be the set
		 of matrices of rank $\leq n-i$. According to \cite[p.~69]{ACGH}, one has
		 $(D_i)_{\sing} = D_{i+1}$, and $D_i$ is exactly the set of points of
		 multiplicity $\geq i$ on $D$. It is proved in \cite[Chapter 4]{Johnson} and also in
		 \cite{Vainsencher} that one can construct a log
		 resolution of the pair $(X,D)$ by successively blowing
		 up $D_n, D_{n-1}, \dotsc, D_2$.
    \item Let $X=\mathbf{P} H^0(C,M)$ and let $D=\Sec^n(C)$ be the $n$-th secant
		 variety of a smooth projective curve $C$, embedded by a line bundle $M$ with
		 $h^0(M) = 2n+3$ that separates $2n+2$ points. Setting $D_i = \Sec^{n-i+1}(C)$,
		 Bertram \cite[Page 440]{Bertram92} proved that $(D_i)_{\sing} = D_{i+1}$ and
		 that $D_i$ is again the set of points of multiplicity $\geq i$ on $D$. He also
		 showed \cite[Corollary 2.4]{Bertram92} that successively blowing up $D_n,
		 D_{n-1}, \dotsc, D_2$ produces a log resolution of the pair $(X,D)$.
\end{enumerate}

The subvarieties in \eqref{eqn: Brill-Noether sequence} induces a stratification  \[ \Jac(C)=(\Jac(C)-\Theta)\sqcup \bigsqcup_{0\leq r\leq n} (W^r_{g-1}(C)-W^{r+1}_{g-1}(C)), \]
which is named the Brill-Noether stratification. Inspired by the proof of Theorem \ref{theorem: log resolution of hyperelliptic theta divisors}, we also find the following.
\begin{prop}\label{prop: Brill Noether are Whitney}
If $C$ is a smooth hyperelliptic curve, then the Brill-Noether stratification of $\Jac(C)$ is a Whitney stratification.
\end{prop}
\endgroup

\subsection*{Ideas of the proof}

The main tool is Bertram's blowup construction for a
chain of maps \cite{Bertram92}. One inconvenient point in the process described above is
that the Brill-Noether varieties $W_{g-1}^r(C)$ are not smooth, which makes it hard
to keep track of conormal bundles and exceptional divisors in the various blowups.
Fortunately, on a hyperelliptic curve, each $W_{g-1}^r(C)$ has a natural resolution
of singularities by $C_{g-1-2r}$, viewed as the space of effective divisors of degree $g-1-2r$ on
$C$. Let $g_2^1$ be the line bundle corresponding to the hyperelliptic map $h
\colon C \to \bP^1$. For $0\leq r\leq n=\lfloor \frac{g-1}{2}\rfloor$, the resolution of singularities is given by the Abel-Jacobi mapping
\[
	\delta_{n-r} : C_{g-1-2r} \to W_{g-1}^r(C), \quad D \mapsto r g_2^1 \otimes \cO_C(D).
\]
Since it is easier to blow up smooth varieties, instead of $W_{g-1}^r(C)$, we work
with the chain of maps $\{\delta_i\}_{i=0}^n$. The advantage is that we do not need
to analyze the singularities of the proper transforms of $W^r_{g-1}(C)$ and how they
intersect with exceptional divisors; instead, we transform the problem into checking
that certain maps are embeddings (see Proposition \ref{prop: injective map imply smooth chain and normal crossing}), which eventually reduces to the calculation of certain
conormal bundles. The projectivized conormal bundles that show up as exceptional
divisors are closely related to secant bundles over symmetric products of $\bP^1$;
for that reason, Bertram's result about these secant bundles is another crucial
input.

\subsection*{Outline of the paper}

In \S \ref{sec: Bertram's blow up construction}, we recall Bertram's blowup
construction in detail. In \S \ref{sec: Abel-Jacobi maps}, we set up notations for the
Abel-Jacobi maps and reduce the proof of Theorem \ref{theorem: log resolution of
hyperelliptic theta divisors} to two propositions (Proposition \ref{prop: chain of
addition maps} and Proposition \ref{prop: chain of Abel Jacobi maps}), which deal
with the properties of two chains of maps between symmetric products and
Jacobians. In \S \ref{sec:Secant bundles and
maps between them}, we review the construction of secant bundles and describe Bertram's
results. In \S \ref{sec:
calculation of conormal bundles}, we study some basic properties of Abel-Jacobi maps
and the addition maps among symmetric products. In \S \ref{sec: proof of chain of abel jacobi maps}-\S
\ref{sec: proof of chain of additions maps}, we prove Proposition \ref{prop: chain of
addition maps} and Proposition \ref{prop: chain of Abel Jacobi maps}, and thereby
complete the proof of Theorem \ref{theorem: log resolution of hyperelliptic theta
divisors} for hyperelliptic curves of odd genus. The proof of Corollary \ref{corollary: generic structure of exceptional divisors} can be found at the end of \S \ref{sec: proof of chain of abel jacobi maps}. 
In \S \ref{sec: even genus
case}, we outline a proof for hyperelliptic curves of even genus, which goes along
the same line but requires a few changes in the notation. In \S \ref{sec: Whitney stratification}, we prove Proposition \ref{prop: Brill Noether are Whitney}. In \S \ref{sec: questions}, we propose some questions in the direction of this paper.

\subsection*{Notation}
\begin{itemize}
	 \item If $V$ is a vector space, $\bP(V)$ stands for the projective space of
		 one-dimensional quotients of $V$. We use the same notation for vector bundles.
    \item Let $f:X\to Y$ be a morphism between smooth projective varieties. Let
		 $Y_1\subseteq Y$ be a subvariety. We use the notation
    \[ f^{-1}(Y_1)\colonequals Y_1\times_Y X,\]
	 exclusively for the scheme-theoretic preimage, which is the fiber product of
	 the two morphisms $X \to Y$ and $Y_1 \to Y$.    
 \item Let $f:X \to Y$ be a morphism between smooth varieties. We denote by
	 \[
		 df : f^* \ctbl{Y} \to \ctbl{X}
	 \]
	 the induced morphism between cotangent bundles, and by	 
	 \begin{equation}\label{eqn: notation of conormal bundle}
		 N_f^{\ast} = \Ker \bigl( df : f^* \ctbl{Y} \to \ctbl{X} \bigr)
	 \end{equation}
	 the conormal bundle of the morphism. In the case of a closed embedding	 $X \subseteq Y$, we also use the notation $N_{X|Y}^{\ast}$. 
	 
\end{itemize} 

\subsection*{Acknowledgement}

We thank Rob Lazarsfeld for suggesting the statement of Theorem \ref{theorem: log resolution of hyperelliptic theta divisors} and we thank
Gavril Farkas, Zhuang He, Carl Lian, Andr\'{a}s L\H{o}rincz, Mirko Mauri and Botong Wang for helpful discussions. We thank Nero Budur for the discussion of the reducedness of $W^r_d(C)$. We also thank the
Max-Planck-Institute for Mathematics for its hospitality and for providing us with
excellent working conditions. Ch.S. was partially
supported by NSF grant DMS-1551677 and by a Simons Fellowship. R.Y. was partially supported by the ERC Advanced Grant SYZYGY.

\section{Bertram's blowup construction}\label{sec: Bertram's blow up construction}

In this section, we review \cite[\S 2]{Bertram92} and provide more details for the benefit of readers. The main result is Proposition \ref{prop: injective map imply smooth chain and normal crossing}, which is an inductive criterion for constructing a log resolution out of a sequence of morphisms.

\subsection{Chains and maps of chains}
Let $X$ be a projective variety, not necessarily
smooth.
\begin{definition}\label{definition: Bertrams blowup construction}
A \emph{proper chain} is a sequence of morphisms $\{ f_i: X_i \to X\}_{i=0}^{n}$ from
projective varieties $X_i$ with the property that for each $0 \leq i<j \leq n$, there
exists a commutative diagram
\begin{equation*}\label{eqn: the diagram with Xij}
\begin{tikzcd}
X_{i,j} \arrow[r,twoheadrightarrow,"g_{i,j}"] \arrow[d,"h_{i,j}"] & X_i \arrow[d,"f_i"] \\
X_j \arrow[r,"f_j"] & X
\end{tikzcd}\end{equation*}
so that $g_{i,j}$ is surjective and there is a proper inclusion $f_i(X_i)\subsetneq f_j(X_j)$. 
\end{definition}

\begin{remark}
It is sufficient to take $X_{i,j}$ to be the fiber product of $f_i$ and $f_j$. However, in practice $X_{i,j}$ is usually not the fiber product and will become a fiber product after sufficiently many blow-ups, see Remark \ref{remark: why Xij is not a fiber product} for one example. \end{remark}

\begin{definition}\label{definition: construction of blow ups}
Let $\{ f_j: X_j \to X\}$ be a proper chain. Assume $f_0$ is an \emph{embedding}, we identify $X_0$ with its image and define:
\begin{align*}
	\bl_1(X) &\colonequals \text{the blowup of $X$ along $X_0$}, \\
	\bl_1(X_j) &\colonequals \text{the blowup of $X_j$ along $f_j^{-1}(X_0)$}, \\
	\bl_1(f_j) &\colonequals \text{the unique lift of $f_j$ to a map
	$\bl_1(X_j) \to \bl_1(X)$}.
\end{align*}
Assume for some $i\geq 1$, $\bl_i(X)$, $\bl_i(X_j)$ and $\bl_i(f_j)$ have already been defined for all $j \geq
i$, and that the map
\[\bl_i(f_i): \bl_i(X_i) \to \bl_i(X)\]
is an \emph{embedding}. Then we can identify
$\bl_i(X_i)$ with its image and set
\begin{align*}
	\bl_{i+1}(X) &\colonequals \text{the blowup of $\bl_i(X)$ along $\bl_i(X_i)$}, \\
	\bl_{i+1}(X_j) &\colonequals \text{the blowup of $\bl_i(X_j)$ along
	$\bl_i(f_j)^{-1}(\bl_i(X_i))$}, \\
	\bl_{i+1}(f_j) &\colonequals \text{the unique lift of $\bl_i(f_j)$ to a map
		$\bl_{i+1}(X_j) \to \bl_{i+1}(X)$}.  
\end{align*}
\end{definition}

\begin{notation}
To have a uniform notation, we set
\[ \bl_0(f_i)\colonequals f_i, \quad \bl_0(X_i) \colonequals X_i, \quad \bl_0(X)
\colonequals X.\]
\end{notation}

\begin{definition}\label{definition: chain of centers}
Let $\{f_i: X_i \to X\}_{i=0}^{n}$ be a proper chain. If $\bl_{n+1}(X)$ is defined in Definition \ref{definition: construction of blow ups},
we say that $\{f_i\}$ is a \emph{chain of centers}. Concretely, this amounts to the
(recursive) condition that the $n+1$ morphisms $f_0$, $\bl_1(f_1)$, \dots,
$\bl_n(f_n)$ are closed embeddings. 
If $X$ and $\{\bl_i(X_i)\}_{i=0}^n$ are all smooth, then we say that $\{f_i \}$ is a
\emph{chain of smooth centers}. \end{definition}

We formulate an additional definition, which is not in \cite{Bertram92} and ensures that the exceptional divisors in
the final blowup $\bl_{n+1}(X)$ form a simple normal crossing divisor. 
\begin{definition}\label{definition: NCD conditions}
Let $\{ f_i: X_i \to X\}_{i=0}^n$ be a chain of smooth centers. We define divisors 
\[ E_{i,j}\subseteq \bl_j(X), \quad \textrm{for $j\geq i+1$},\]
as follows. First $E_{i, i+1}\subseteq \bl_{i+1}(X)$ is the exceptional divisor for the blowing-up
of $\bl_i(X)$ along $\bl_i(X_i)$. For each $j \geq i+1$, let $E_{i,j}\subseteq \bl_{j}(X)$
be the inverse image of $E_{i,i+1}$, with the following Cartesian diagrams:
\begin{equation*}
\begin{tikzcd}
E_{i,j} \arrow[d,hook] \arrow[r] & E_{i,i+1} \arrow[d,hook] \arrow[r] & \bl_i(X_i) \arrow[d,hook]\\
\bl_{j}(X) \arrow[r] & \bl_{i+1}(X) \arrow[r] & \bl_i(X) 
\end{tikzcd} 
\end{equation*}
Moreover, set
\[ E_{i}\colonequals E_{i,n+1}\subseteq \bl_{n+1}(X),\]
and we say that $\{E_{i}\}_{i=0}^n$ is \emph{the set of exceptional divisors} of the chain $\{f_i\}_{i=0}^n$.

A chain of smooth centers $\{ f_i \}_{i=0}^n$ is called an \emph{NCD chain} if, for each $j\leq n$, the intersection
\[ \bl_j(X_j) \cap E_{i,j} \subseteq \bl_{j}(X)\]
is transverse for all $i<j$ and the divisor $E_{0,j}+\cdots +E_{j-1,j} \subseteq \bl_j(X)$
has simple normal crossings. 
\end{definition}

\begin{remark}\label{remark: complement of union of divisors} 
Let $\{f_i:X_i\to X\}_{i=0}^{n}$ be a chain of
smooth centers. Then there is a natural embedding of $X-f_{n}(X_n)$ into
$\bl_{n+1}(X)$ such that
\[ \bl_{n+1}(X)-\bigcup_{0\leq i\leq n}E_{i}=X-f_n(X_{n}).\]
\end{remark}

For induction purposes, the following notation becomes convenient.
\begin{notation}\label{notation: chain product with a constant factor}
Let $S$ be a smooth variety and let $\{f_i:X_i \to X\}_{i=0}^n$ be a proper chain. It
induces a new proper chain $\{f_i\times \id: X_i \times S \to X\times S \}_{i=0}^n$, i.e.
the collection of maps that are $f_i$ on the first factor and the identity on $S$.
\end{notation}
\begin{lemma}\label{lemma: chains with a constant factor}
Let $\{f_i:X_i\to X\}_{i=0}^n$ be an NCD chain. For each $0 \leq i < j \leq n+1$, let
$E_{i,j}\subseteq \bl_j(X)$ and $F_{i,j} \subseteq \bl_j(X\times S)$ be the
exceptional divisors of the chains $\{f_i\}_{i=0}^n$ and $\{f_i\times
\id\}_{i=0}^n$. Then there are natural isomorphisms
\[ \bl_i(X\times S)=\bl_i(X)\times S, \quad \bl_i(X_j\times S)=\bl_i(X_j)\times S,  \]
\[ \bl_i(f_j\times \id)=\bl_i(f_j)\times \id, \quad F_{i,j}=E_{i,j}\times S.\]
Moreover, $\{f_i\times \id\}_{i=0}^n$ is again an NCD chain.
\end{lemma}

\begin{proof}
This follows from the fact that blowup maps commute with taking Cartesian product
with the smooth variety $S$, and that transversality is also preserved under product with $S$.
\end{proof}

To prove that a chain of centers is an NCD chain, it is useful to have the following notion.

\begin{definition}\label{definition: injective chain}
Suppose that $\{ f_i : X_i \to X\}_{i=0}^{n}$ and $\{ g_i: Y_i \to Y\}_{i=0}^{n}$ are
two chains of centers. We say that a map $\phi: X \to Y$ is a \emph{map of chains of
centers} if it satisfies the following conditions.
\begin{itemize}
\item First $\phi^{-1}(Y_0)=X_0$, so $\bl_1(X) \to \bl_1(Y)$ is defined, and
\item inductively, assume that for some $0<i\leq n$ the map $\bl_i(\phi):\bl_i(Y)\to \bl_i(X)$ exists, then one has $\bl_i(\phi)^{-1}(\bl_i(Y_i))=\bl_i(X_i)$. Consequently, one can define
\[ \bl_{i+1}(\phi):\bl_{i+1}(X) \to \bl_{i+1}(Y) \]
to be the unique lift of $\bl_i(\phi)$.\end{itemize}
We say that the map $\phi$ is an \emph{injective} map of chains of centers if in addition, $\bl_{n+1}(\phi)$ is injective.
\end{definition}

The following lemma gives an inductive criterion for a chain to be NCD.
\begin{lemma}\label{lemma: injective map imply normal crossing}
	Let $\{\phi_j: X_j \to X\}_{j=0}^n$ be a proper chain, with diagrams for $0\leq i<j \leq n$:
\[ \begin{tikzcd}
X_{i,j} \arrow[d,"f_{i,j}"]\arrow[r,twoheadrightarrow] & X_i\arrow[d,"\phi_i"]\\
X_j \arrow[r,"\phi_j"] & X
\end{tikzcd}\]
Suppose that $X$, $X_i$, $X_{i,j}$ are all smooth projective varieties and assume the following conditions are satisfied:
\begin{enumerate}
\item For each $j\leq n$, $\{f_{i,j}:X_{i,j} \to
X_{j}\}_{i=0}^{j-1}$ is an NCD chain.
\item $\{\phi_j\}_{j=0}^n$ is a chain of smooth centers.
\item Each $\phi_j$ is a map of chains of centers.
\end{enumerate}
Then the chain $\{\phi_j\}_{j=0}^n$ is an NCD chain.
\end{lemma}

\begin{proof}
It follows from \cite[Lemma 2.1]{Bertram92} and its proof.\end{proof}

In the rest of this section, we discuss how one can replace the conditions (2) and (3) in Lemma
\ref{lemma: injective map imply normal crossing} by certain conditions on exceptional divisors and their complements, in the presence of condition (1); more precisely the transversality condition from $\{f_{i,j}\}$ being NCD chains. By extracting these conditions from the proof of \cite[Proposition 2.2]{Bertram92}, we hope it will make our proof of Theorem \ref{theorem: log resolution of hyperelliptic theta divisors} more transparent.

\subsection{Criteria for maps of chains of centers}
For a map $\phi:X\to Y$ between chains $\{f_i:X_i\to X\}$ and $\{g_i:Y_i \to Y\}$, to show it is a map of chains of centers, one needs to check $\bl_i(\phi)^{-1}(\bl_i(Y_i))=\bl_i(X_i)$ for each $i$. In the presence of certain transversality conditions, we can apply the lemma below. 

\begin{notation}\label{notation: open parts of divisors}
For a sequence of divisors $\{E_i\}_{i=0}^{n}$, denote by
\[ E^{\circ}_i\colonequals E_i-(E_0 \cup E_1 \cup \dotsb \cup E_{i-1}). \]
Note that we are removing \emph{only} the intersections with the previous divisors.
\end{notation}
\begin{lemma}\label{lemma: criterion for checking scheme-theoretic pull back}
Let $\phi:X\to Y$ be a morphism between smooth projective varieties. Let $\{
E_i\}_{i=0}^{j-1}$ and $\{ F_i\}_{i=0}^{j-1}$ be two sequences of smooth divisors in
$X$ and $Y$ such that 
\[\phi^{-1}(F_i)=E_i, \quad \forall 0\leq i \leq j-1.\]
Let $X_1\subseteq X$ and $Y_1\subseteq Y$ be two smooth subvarieties. Assume that for
each $0\leq i \leq j-1$,
\begin{itemize}
    \item [(a)] the intersections $X_1\cap E_i$ and $Y_1 \cap F_i$ are transverse,
    \item [(b)] $\phi^{-1}(Y_1\cap F^{\circ}_i)=X_1\cap E^{\circ}_i$,
    \item [(c)] $\phi^{-1}(Y_1-\cup_i F_i)=X_1-\cup_i E_i$. 
\end{itemize}
Then $\phi^{-1}(Y_1)=X_1$.
\end{lemma}

\begin{proof}
From the assumption, we know that the set-theoretic preimage of $Y_1$ under $\phi$ is $X_1$. 
In order to show the scheme-theoretical statement, we need to know that
$\phi^{\ast} \cI_{Y_1} \to \cI_{X_1}$ is surjective. Recall that for a closed embedding of smooth varieties $A\xhookrightarrow{i}  B$, the conormal bundle is denoted by
\[ N^{\ast}_{A|B}= \mathrm{Ker}\{ di: i^{\ast}T^{\ast}_B \to T^{\ast}_A\}.\] Since $N_{X_1|X}^{\ast} =
\cI_{X_1}/\cI_{X_1}^2$ and $N_{Y_1|Y}^{\ast} = \cI_{Y_1}/\cI_{Y_1}^2$, by Nakayama's
lemma, this is equivalent to the surjectivity of
\[ d\phi: \phi^{\ast}N^{\ast}_{Y_1|Y} \to N^{\ast}_{X_1|X}.\]
This can be checked over $X_1-\cup_i E_i=X_1-\cup_iE^{\circ}_i$ and $X_1\cap E^{\circ}_i$ separately. Condition (c) implies that $d\phi$ is surjective over $X_1-\cup_i E_i$. On the other hand, using condition (a), we have the following commutative diagram
\[ \begin{tikzcd}[column sep=large] 
\phi^{\ast}N^{\ast}_{Y_1|Y} \big\vert_{Y_1\cap F^{\circ}_i} \arrow[r,"d\phi
\vert_{X_1\cap E^{\circ}_i}"] \arrow[d,"\cong"] & N^{\ast}_{X_1|X} \big\vert_{X_1\cap E^{\circ}_i}\arrow[d,"\cong"] \\
\phi^{\ast}N^{\ast}_{Y_1\cap F^{\circ}_i | F^{\circ}_i} \arrow[r,"d(\phi
\vert_{E^{\circ}_i})"] & N^{\ast}_{X_1\cap E^{\circ}_i|E^{\circ}_i}.
\end{tikzcd}\]
The bottom map is induced by $\phi \vert_{E^{\circ}_i}: E^{\circ}_i \to
F^{\circ}_i$. The vertical maps are isomorphisms because of the transversality
condition. Therefore condition (b) implies that $d\phi$ is also
surjective over $X_1\cap E^{\circ}_i$ for each $i$.
\end{proof}

\subsection{Criterion for a chain of centers}
Let $\{h_i:Z_i\to Z\}$ be a proper chain where $Z_i$ and $Z$ are all smooth. Assume for some $i$, the map 
\begin{equation}\label{eqn: blow up map from Zi to Z}
\bl_i(h_i):\bl_i(Z_i) \to \bl_i(Z)\end{equation}
exists. In this section, we explain how to check it is an embedding under suitable transversality conditions.   
\begin{lemma}\label{lemma: pull back of conormal bundles}
Let $f: X \to Y$ be a morphism between two smooth projective varieties. Let
$F\subseteq Y$ be a smooth divisor. If $E\colonequals f^{-1}(F)$ is also a smooth
divisor, then
\[ f^{\ast}N^{\ast}_{F|Y}=N^{\ast}_{E|X},\]
where $N^{\ast}_{E|X}$ is the conormal bundle of $E$ in $X$.
\end{lemma}

\begin{proof}
Since $F$ is a smooth divisor, we have $N^{\ast}_{F|Y}=\cO_{Y}(-F) \vert_F$ by the conormal
sequence; the same holds for $E$. Therefore
\[ f^{\ast}N^{\ast}_{F|Y}=f^{\ast}\cO_Y(-F) \vert_F=\cO_{X}(-E) \vert_E=N^{\ast}_{E|X}. \qedhere \]
\end{proof}

\begin{lemma}\label{lemma: criterion for embedding of blowups}
Let $f:X \to Y$ be a morphism between two smooth projective varieties. Let $Z\subseteq
Y$ be a smooth subvariety such that $W\colonequals f^{-1}(Z)$ is smooth and
properly contained in $X$. Denote by $\tilde{Y} = \bl_Z Y$ and $\tilde{X} = \bl_W
X$ with the following diagram
\[ \begin{tikzcd}
		E \arrow[r,hookrightarrow] \arrow[d]& \tilde{X} \arrow[r,"\tilde{f}"] \arrow[d] &
		\tilde{Y} \arrow[d] \arrow[r,hookleftarrow] & F\arrow[d]\\
W\arrow[r,hook] & X \arrow[r,"f"] & Y \arrow[r,hookleftarrow] &Z
\end{tikzcd}
\]
where $\tilde{f} : \tilde{X} \to \tilde{Y}$ is the
induced morphism and $F$ and $E$ are the exceptional divisors. Then 
\begin{equation}\label{eqn: functoriality of blow ups}
	\tilde{f}^{-1}(F)=E, \quad \tilde{f}^{\ast}N^{\ast}_{F|\tilde{Y}}=N^{\ast}_{E|\tilde{X}}.
\end{equation}
Furthermore, suppose
\begin{itemize}
	\item[(a)] the map $\tilde{f}\big\vert_E: E \to F$ is an embedding,
    \item[(b)] the map $f:X-W \to Y-Z$ is an embedding.
\end{itemize}
Then $\tilde{f}$ is an embedding. 
\end{lemma}

\begin{proof}
It is proved in \cite[Page 442, Fact A]{Bertram92} that $\tilde{f}^{-1}(F)=E$. By assumption, $Z$ and $W$ are smooth, therefore $E$ and $F$ are smooth divisors and we can apply Lemma \ref{lemma: pull back of conormal bundles} to obtain \eqref{eqn: functoriality of blow ups}.

By condition (b),  $\tilde{f}$ is an embedding away from $E$. Condition (a) implies
that $\tilde{f}$ is set-theoretically injective over $E$. Therefore, it suffices to show
that \[d\tilde{f}:\tilde{f}^{\ast} \ctbl{\tilde{Y}} \to \ctbl{\tilde{X}}\]
is surjective over $E$.  Consider the following commutative diagram:
\[ \begin{tikzcd}
		0 \arrow[r] & \tilde{f}^{\ast}N^{\ast}_{F|\tilde{Y}} \arrow[r] \arrow[d,"
		\cong"] &
		\tilde{f}^{\ast} \ctbl{\tilde{Y}} \big\vert_{E} \arrow[r] \arrow[d,"d\tilde{f}"] &
		\tilde{f}^{\ast} \ctbl{F} \arrow[r] \arrow[d,"d(\tilde{f}|_E)"] & 0\\
		0 \arrow[r] & N^{\ast}_{E|\tilde{X}} \arrow[r] & \ctbl{\tilde{X}} \big\vert_{E}
		\arrow[r] & \ctbl{E} \arrow[r] & 0
\end{tikzcd}
\]
By \eqref{eqn: functoriality of blow ups}, the arrow
on the left is an isomorphism. Since $\tilde{f} \big\vert_{E}$ is an embedding, the
arrow on the right is surjective.  By the snake lemma, we conclude that the arrow in
the middle is also surjective, and conclude that $\tilde{f}$ is an
embedding. .
\end{proof}

In the situation of \eqref{eqn: blow up map from Zi to Z}, to prove $\bl_i(h_i)$ is an embedding, usually we also know $h_i:Z_i \to Z$ is a map of chains of centers and thus we give a generalization of Lemma \ref{lemma: criterion for embedding of blowups} in such situations. Let $\{ f_i: X_i \to X\}_{i=0}^{j-1}$ and $\{g_i:Y_i \to
Y\}_{i=0}^{j-1}$ be two chains of centers, and let $\phi: X \to Y$ be a map of chains
of centers, so that the map $\bl_i(\phi): \bl_i(X) \to \bl_i(Y)$
exists for each $i\leq j$ and 
\begin{equation}\label{eqn: conditions for map of chains of centers}
    \bl_{i}(\phi)^{-1} \bigl( \bl_i(X_i) \bigr)=\bl_i(Y_i), 
	 \quad \forall 0\leq i \leq j-1.
\end{equation} 
For each $i$, consider the following diagram:
\begin{equation}\label{eqn: diagram of exceptional divisors}
     \begin{tikzcd}
E_i \arrow[r,hookrightarrow] \arrow[d] & \bl_j(X) \arrow[r,"\bl_j(\phi)"]
\arrow[d] & \bl_j(Y) \arrow[d] \arrow[r,hookleftarrow] & F_i \arrow[d]\\
\bl_i(X_i)\arrow[r,hookrightarrow] & \bl_{i}(X) \arrow[r,"\bl_{i}(\phi)"] &
\bl_{i}(Y)  \arrow[r,hookleftarrow] & \bl_i(Y_i)
\end{tikzcd}
\end{equation}
Here $E_i,F_i$ are exceptional divisors. Repeatedly applying \eqref{eqn: functoriality of blow ups} and \eqref{eqn: conditions for map of chains of centers}, one has
\begin{equation}\label{eqn: pull back of Fi is Ei}
    \bl_j(\phi)^{-1}(F_i)=E_i, \quad \forall 0\leq i \leq j-1.
\end{equation} 

\begin{lemma}\label{lemma: embedding for multiple blow ups}
Assume $X,Y$ are smooth projective varieties, and
\begin{itemize} 
	\item[(a)] the two chains $\{ f_i\}_{i=0}^{j-1},\{ g_i\}_{i=0}^{j-1}$ are NCD chains,
    \item[(b)] for every $0 \leq i \leq j-1$, the induced  map
		 \[ \bl_j(\phi): E^{\circ}_i\to F^{\circ}_i \]
is an embedding, where $E_i^{\circ}=E_i-\cup_{\ell<i} E_{\ell}$, the same for $F_i^{\circ}$,
	 \item[(c)] and the map $\phi: X-f_{j-1}(X_{j-1}) \to Y$ is an embedding.
\end{itemize}
Then $\bl_j(\phi)$ is an embedding. 
\end{lemma}

\begin{proof}
The transversality condition (a) guarantees that $F_i$ and $E_i$ are smooth divisors in smooth projective varieties. Together with \eqref{eqn: pull back of Fi is Ei}, we can apply Lemma \ref{lemma: pull back of conormal bundles} to obtain
\begin{equation*}
    \bl_j(\phi)^{\ast}N^{\ast}_{F_i|\bl_j(Y)}=N^{\ast}_{E_i|\bl_j(X)}.
\end{equation*}
Restricting to $E^{\circ}_i$ and $F^{\circ}_i$ gives
\begin{equation}\label{eqn: identification of conormal bundles of blowups on opens}
    \bl_j(\phi)^{\ast}N^{\ast}_{F^{\circ}_i|\bl_j(Y)}
	 =N^{\ast}_{E^{\circ}_i|\bl_j(X)}.
\end{equation}  
Consider the following commutative diagram:
\[ \begin{tikzcd}
0 \arrow[r] & \bl_j(\phi)^{\ast}N^{\ast}_{F^{\circ}_i|\bl_j(Y)} \arrow[r]
\arrow[d,"\cong"] & \bl_j(\phi)^{\ast} \ctbl{\bl_j(Y)}
\big\vert_{E^{\circ}_i} \arrow[r] \arrow[d,"d\bl_j(\phi)"] &
\bl_j(\phi)^{\ast} \ctbl{F^{\circ}_i} \arrow[r] \arrow[d,"d(\bl_j(\phi)
\vert_{E^{\circ}_i})"] & 0\\
0 \arrow[r] & N^{\ast}_{E^{\circ}_i|\bl_j(X)} \arrow[r] &
\ctbl{\bl_j(X)}\big\vert_{E^{\circ}_i} \arrow[r] & \ctbl{E^{\circ}_i} \arrow[r] & 0
\end{tikzcd}
\]
Using \eqref{eqn: identification of conormal bundles of blowups on opens}, the condition (b) and the snake lemma, we see that
\[ d\bl_j(\phi): \bl_j(\phi)^{\ast} \ctbl{\bl_j(Y)} \to \ctbl{\bl_j(X)} \]
is surjective over each $E^{\circ}_i$. By Remark \ref{remark: complement of union of
divisors}, the set $X-f_{n}(X_{n})$ naturally embeds into $\bl_j(X)$ with
complement $\cup_{i}E_i$, and condition (c) says that  $d\bl_j(\phi)$ is
surjective away from $\cup_{i}E_i$. Since $\bigcup_{i}E_i=\bigcup_i E^{\circ}_i$,
we conclude that $\bl_j(\phi)$ is an embedding.
\end{proof}

\begin{remark}
	In condition (b), we can ask $\bl_j(\phi): E_i \to F_i$ to be
an embedding, but in practice it is much easier to check this condition on open
subsets inductively.
\end{remark}

\subsection{Criterion for a proper chain to be NCD}

Putting everything together, we have the following inductive criterion for a NCD chain.
\begin{prop}\label{prop: injective map imply smooth chain and normal crossing} 
Let $\{\phi_j:X_j\to X\}_{j=0}$ be a proper chain with diagrams for each $j<k$
\[ \begin{tikzcd}
X_{j,k} \arrow[d,"f_{j,k}"]\arrow[r,twoheadrightarrow] & X_j\arrow[d,"\phi_j"]\\
X_k \arrow[r,"\phi_k"] & X
\end{tikzcd}\]
Let $k$ be an integer. Assume the following conditions are satisfied.
\begin{enumerate}[(I)]
\item The chains $\{\phi_j\}_{j=0}^{k-1}$ and $\{f_{j,k}\}_{j=0}^{k-1}$ are NCD.
\item The map $\phi_k:X_k-f_{k-1,k}(X_{k-1,k})\to X$ is an embedding and $\phi_k^{-1}(X_0)=X_{0,k}$.
\item Let $j< k$. Suppose the blowup spaces $\bl_{i}(X_k)$ associated to $\{f_{i,k}\}_{i=0}^{j-1}$ and $\{\phi_i\}_{i=0}^{j-1}$ coincide for all $i$ so that the following diagram exists for $i<j$
\begin{equation*}
\begin{tikzcd}
{} & {} & \bl_j(X_j) \arrow[d,"\bl_j(\phi_j)"]& {}\\
E_{i}\arrow[r,hookrightarrow] \arrow[d] &\bl_j(X_k) \arrow[r,"\bl_j(\phi_k)"] \arrow [d] &\bl_j(X) \arrow[d]\arrow[r,hookleftarrow] &F_{i}\arrow[d]\\
\bl_i(X_{i,k}) \arrow[r,hookrightarrow] &\bl_i(X_k) \arrow[r,"\bl_i(\phi_k)"] &\bl_i(X) \arrow[r,hookleftarrow] &\bl_i(X_i) 
\end{tikzcd}
\end{equation*}
where $E_{i},F_{i}$ are exceptional divisors, then 
\begin{align*}
\bl_j(\phi_k)^{-1}(\bl_j(X_j)-\cup_{i<j} F_{i})=\bl_j(X_{j,k})-\cup_{i<j} E_i, \tag{$\ast$}\\
\bl_j(\phi_k)^{-1}( \bl_j(X_j)\cap F_{i}^{\circ})=\bl_j(X_{j,k})\cap E_{i}^{\circ}, \quad \forall i<j\tag{$\ast\ast$}.\\
\end{align*}
Moreover, suppose the space $\bl_k(\phi_k)$ for the chains $\{f_{i,k}\}_{i=0}^{k-1}$ and $\{\phi_i\}_{i=0}^{k-1}$ coincide, then
\begin{align*} 
\bl_k(\phi_k):E_{i}^{\circ}\to F_{i}^{\circ} \textrm{ is an embedding}, \quad \forall i<k,\tag{$\ast\ast\ast$}
\end{align*}
where $E_i^{\circ}=E_i-\cup_{h<i}E_h$.

\end{enumerate}

Then $\phi_k$ is a map of chains of centers and the chain $\{\phi_j\}_{j=0}^k$ is a NCD chain.
\end{prop}
\begin{remark}
The Condition (I) is inductive. The condition (II) are properties of the proper chains and the condition (III) are certain compatibility conditions of blowups associated to the chains $\{f_{i,k}\}_{i=0}^{k-1}$ and $\{\phi_i\}_{i=0}^{k-1}$.
\end{remark}
\begin{proof}
It suffices to prove the following statements.
\begin{enumerate}[(i)]
\item $\bl_j(\phi_k)^{-1}(\bl_j(X_j))= \bl_j(X_{j,k})$ for all $j<k$.
\item $\bl_k(X_k)$ is smooth.
\item The map $\bl_k(\phi_k):\bl_k(X_k)\to \bl_k(X)$ is an embedding.
\end{enumerate}

We prove statement (i) by induction on $j$. The base case follows from condition (II). Suppose it holds for all $i<j$, i.e.
\begin{equation*}
\bl_i(\phi_k)^{-1}(\bl_i(X_i))=\bl_i(X_{i,k}), \quad \forall i<j.
\end{equation*}
It follows that $\bl_j(\phi_k)^{-1}(F_i)=E_i$ for all $i<j$. Furthermore, this says that the space $\bl_{i}(X_k)$ associated to $\{f_{i,k}\}_{i=0}^{j-1}$ and $\{\phi_i\}_{i=0}^{j-1}$ coincide for all $i\leq j-1$, so that the conditions $(\ast)$ and $(\ast\ast)$ hold by (III). On the other hand, Condition (I) implies that $\bl_j(X_j),\{E_i\}_{i=0}^{j-1}$ and $\bl_j(X_{j,k}),\{F_i\}_{i=0}^{j-1}$ are smooth subvarieties of $\bl_j(X)$ and $\bl_j(X_k)$, respectively. Using conditions (I),$(\ast),(\ast\ast)$, we can apply Lemma \ref{lemma: criterion for checking scheme-theoretic pull back} to the map $\bl_j(\phi_k)$ to conclude that 
\[ \bl_j(\phi_k)^{-1}(\bl_j(X_j))= \bl_j(X_{j,k}).\]
This finishes the inductive proof.

Statement (i) implies that the space $\bl_k(\phi_k)$ for the chains $\{f_{i,k}\}_{i=0}^{k-1}$ and $\{\phi_i\}_{i=0}^{k-1}$ coincide, so that the condition $(\ast\ast\ast)$ holds. Since the chain $\{f_{j,k}\}_{j=0}^{k-1}$ is NCD by (I), $\bl_k(X_k)$ must be smooth. Using conditions (I), (II),$(\ast\ast\ast)$, Lemma \ref{lemma: embedding for multiple blow ups} implies that $\bl_k(\phi_k)$ is an embedding. We finish the proof.
\end{proof}

Since the exceptional divisors in smooth blowups are projective bundles, to verify the
assumptions $(\ast\ast)$ and $(\ast\ast\ast)$ in Proposition \ref{prop: injective map imply smooth chain and normal crossing}  in practice, we need relative versions of some lemmas in previous subsections.

\begin{lemma}\label{lemma: fiberwise embedding,pullback etc}
Let $\phi:X\to Y$ be a $B$-morphism of smooth algebraic varieties over a smooth
variety $B$, such that $f$ and $g$ are smooth morphisms:
\[ \begin{tikzcd}
X \arrow[dr,"f"] \arrow[rr,"\phi"]& & Y\arrow[dl,"g"]\\
& B &
\end{tikzcd}
\]
Denote the induced map over a closed point $t\in B$ by $\phi_t:X_t\to
Y_t$. Then the following hold.
\begin{enumerate}
    \item If $\phi_t$ is an embedding for each $t$, then $\phi$ is an embedding.
    \item Let $X_1\subseteq X$ and $Y_1\subseteq Y$ be smooth subvarieties. If $Y_t\cap Y_1,X_t\cap X_1$ are smooth and $\phi_t^{-1}(Y_t\cap Y_1)=X_t\cap X_1$ for each $t$, then $\phi^{-1}(Y_1)=X_1$. 
    \item Suppose $\phi$ is an embedding, then $\bl_X(Y)$ is a $B$-variety and the fiber over $t\in B$ is $\bl_{X_t}(Y_t)$.
\end{enumerate}
\end{lemma}

\begin{proof}
	For (1), it suffices to check that the differential $d\phi: \phi^{\ast} \ctbl{Y}
	\to \ctbl{X}$ is surjective, which can be checked by restriction to each $X_t$; the
proof is similar to that of Lemma \ref{lemma: criterion for embedding of blowups} by using
the isomorphism
\[ \phi^{\ast} N^{\ast}_{Y_t|Y} \cong N^{\ast}_{X_t|X}.\]
For (2), it suffices to check the surjectivity of the induced map
\[ d\phi: \phi^{\ast}N^{\ast}_{Y_1|Y}\to N^{\ast}_{X_1|X}.\]
The proof is similar to that of Lemma \ref{lemma: criterion for checking scheme-theoretic pull back} and uses the isomorphism
\[ N^{\ast}_{Y_1|Y} \big\vert_{Y_1\cap Y_t}\cong N^{\ast}_{Y_1\cap Y_t|Y_t} \]
because $Y_1\cap Y_t$ is smooth; likewise for $N^{\ast}_{X_1|X}$. (3) follows from
local computations.
\end{proof}

\section{Reduction of the proof of Theorem \ref{theorem: log resolution of
hyperelliptic theta divisors}}\label{sec: Abel-Jacobi maps}

In this section, we reduce the proof of Theorem \ref{theorem: log resolution of
hyperelliptic theta divisors} to two propositions about certain natural proper chains being NCD chains, using the  framework developed in \S\ref{sec: Bertram's blow
up construction}. Let $C$ be a hyperelliptic curve of odd genus $g=2n+1$. The even
genus case is similar, and will be treated separately in \S\ref{sec: even genus case}. 
Let $g^1_2$ be the line bundle corresponding to the hyperelliptic map $h : C\to \bP^1$
and denote by $C_j\colonequals \Sym^j(C)$ the $j$-th symmetric product of $C$; we
view the closed points of $C_j$ as effective divisors of degree $j$ on the curve $C$,
and let $C_0$ be the scheme parametrizing the trivial divisor on $C$.
Consider the following sequence of morphisms 
\begin{equation}\label{eqn: the chain of maps to Jacobian}
	\begin{split}
    \delta_j: C_{2j} &\to \Jac(C)=\mathrm{Pic}^{g-1}(C), \quad 0\leq j \leq n,\\
    D &\mapsto (n-j)g^1_2\otimes \cO_{C}(D).
 \end{split}
 \end{equation}
By the Abel-Jacobi theorem, we have $\delta_{j}(C_{2j})=W^{n-j}_{g-1}$ and a
natural embedding 
\[\bP^{j}=\delta_{j}^{-1}(ng^1_2)\hookrightarrow C_{2j}.\]
For $i<j$, we have a commutative diagram
\begin{equation}\label{eqn: diagram for abel jacobi maps}
\begin{tikzcd}
C_{2i} \times \bP^{j-i} \arrow[r,twoheadrightarrow,"p_1"] \arrow[d,"\gamma_{i,j}"] & C_{2i}\arrow[d,"\delta_{i}"]\\
C_{2j} \arrow[r,"\delta_j"] & \Jac(C),
\end{tikzcd}
\end{equation} 
where $p_1$ denotes the projection to the first factor and 
\begin{equation}\label{eqn: the chain of maps to symmetric products}
    \gamma_{i,j}: C_{2i}\times \bP^{j-i} \hookrightarrow C_{2i}\times C_{2j-2i} \to C_{2j}, \quad 0\leq i \leq j-1,
\end{equation}
is the composition of the natural embedding $\bP^{j-i}\hookrightarrow C_{2j-2i}$ and the addition map on symmetric products. Since 
\[\delta_{i}(C_{2i})=W^{n-i}_{g-1}\subsetneq W^{n-j}_{g-1}=\delta_{j}(C_{2j}),\]
the chain $\{\delta_j\}_{j=0}^{n}$ is proper.

To understand the chain $\{\delta_j\}_{j=0}^{n}$, we need several auxiliary chains. The first such chain is
\begin{equation}\label{eqn: chain gamma ik}
\{\gamma_{i,k}:C_{2i}\times \bP^{k-i}\to C_{2k}\}_{i=0}^{k-1},
\end{equation}
and for $i<j<k$, we have the following commutative
diagram:
  \begin{equation}\label{eqn: diagram for addition maps}
  \begin{tikzcd}
    C_{2i}\times \bP^{j-i} \times \bP^{k-j} \arrow
	 [r,twoheadrightarrow,"\mathrm{id}\times r"] \arrow[d,"\gamma_{i,j}\times \id"] & C_{2i} \times \bP^{k-i} \arrow[d,"\gamma_{i,k}"]\\
    C_{2j}\times \bP^{k-j} \arrow[r,"\gamma_{j,k}"] & C_{2k} 
    \end{tikzcd}
    \end{equation}
Here $r$ is the restriction of the addition map $C_{2(j-i)}\times C_{2(k-j)} \to
C_{2(k-i)}$, which coincides with the addition map for symmetric products of
$\bP^1$ by thinking $\bP^{\ell}$ as $\Sym^{\ell}\bP^1$. It is easy to see that $\gamma_{i,k}(C_{2i}\times \bP^{k-i})$ parametrizes
effective divisors $D$ of degree $2k$ such that 
$h^0(\cO_C(D))\geq k-i+1$. It follows that
\[ \gamma_{i,k}(C_{2i}\times \bP^{k-i})\subsetneq \gamma_{j,k}(C_{2j}\times \bP^{k-j}), \quad i<j,\]
and so $\{\gamma_{i,k}\}_{i=0}^{k-1}$ is a proper chain.

The second auxiliary chain lives over $C_{2j} \times \bP^{k-j}$ for a fixed tuple $(j,k)$ with $j<k$. Consider the chain 
\[\bigl\{\gamma_{i,j}\times
	\id:(C_{2i}\times \bP^{j-i})\times \bP^{k-j}\to
C_{2j}\times \bP^{k-j}\bigr\}_{i=0}^{j-1}
\]
which is induced by taking the product of the chain $\{\gamma_{i,j}\}_{i=0}^{j-1}$
with $\bP^{k-j}$, as in Notation \ref{notation: chain product with a constant factor}. Fortunately, because of the product structure of this chain, the inductive
process stops here and no further auxiliary chains are needed.

Denote by $\bl_i(C_{2j}\times \bP^{j-k})$ and $\bl_i(C_{2k})$ the spaces associated to the chain $\{\gamma_{j,k}\}_{j=0}^{k-1}$. Denote by $\bl_i(C_{2i}\times \bP^{j-i}\times \bP^{k-j})$ the spaces associated to the chain $\{\gamma_{i,j}\times \mathrm{id}\}_{i=0}^{j-1}$. In \S \ref{sec: proof of chain of additions maps}, the following will be proved.

\begin{prop}\label{prop: chain of addition maps}
Let $k$ be an integer such that $1\leq k\leq n$. Then the proper chain $\{\gamma_{j,k}\}_{j=0}^{k-1}$
is an NCD chain and for each $j<k$, the map
\[  \gamma_{j,k}: C_{2j}\times \bP^{k-j} \to C_{2k}\]
is a map of chains of centers from $\{\gamma_{i,j} \times \id\}_{i=0}^{j-1}$ to
$\{\gamma_{i,k}\}_{i=0}^{j-1}$. Concretely, this means the following things:
\begin{itemize}
	\item[(a)] For $0 \leq i < k$, there is a closed embedding	\begin{equation*}\label{eqn:bli C2itimes Pk}
	\bl_i(\gamma_{i,k}):\bl_i(C_{2i} \times
		\bP^{k-i}) \hookrightarrow
		\bl_i(C_{2k}),\end{equation*}
		whose
		image intersects the union of all the exceptional divisors in $\bl_i(C_{2k})$
		transversely.
	\item[(b)] There is a natural embedding $C_{2k}-\gamma_{k-1,k}(C_{2k-2}\times
		\bP^1)\hookrightarrow \bl_{k}(C_{2k})$, whose complement has $k$ smooth
		components with normal crossings.
	\item[(c)] $\bl_i(C_{2j}\times \bP^{j-k})$ coincides with the space associated to $\{\gamma_{i,j}\times \mathrm{id}\}_{i=0}^{j-1}$ and one has a natural identification
 \[ \bl_i \bigl( C_{2i} \times \bP^{j-i}  \times \bP^{k-j}\bigr)=\bl_i(C_{2i})\times \bP^{j-i}  \times \bP^{k-j},\]
so that for each $i<j<k$, one has a Cartesian diagram
	\[ \begin{tikzcd}
			\bl_i ( C_{2i}) \times \bP^{j-i}  \times \bP^{k-j}
		\dar[hook]{\bl_i(\gamma_{i,j})\times \id} \rar{\id \times r} & 
		\bl_i (C_{2i}) \times \bP^{k-i}  \dar[hook]{\bl_i(\gamma_{i,k})} \\
		\bl_i(C_{2j}) \times \bP^{k-j} \rar{\bl_i(\gamma_{j,k})} & \bl_i(C_{2k})
 	\end{tikzcd}, \]
  where $r:\bP^{j-i}  \times \bP^{k-j}\to \bP^{k-i}$ is the addition map.
\end{itemize}

\end{prop}

With the help of the proposition above, in \S\ref{sec: proof of chain of abel jacobi maps} we will prove the following proposition. 

\begin{prop}\label{prop: chain of Abel Jacobi maps}
The proper chain $\{\delta_{j}:C_{2j}\to \Jac(C)\}_{j=0}^{n}$ is an NCD chain and the map $\delta_k:C_{2k}\to \Jac(C)$ is a map of chains of centers between the chains $\{\gamma_{j,k}\}_{j=0}^{k-1}$ and $\{\delta_j\}_{j=0}^{k-1}$. Concretely, this means the following. 
\begin{itemize}
	\item[(a)] For $0\leq j \leq n$, the maps $\bl_j(\delta_k)$ associated to the chain $\{\delta_j\}_{j=0}^n$ exist and 
 \[ \bl_j(\delta_j):\bl_j(C_{2j}) \to \bl_j(\Jac(C))\]
 is an embedding so that its image intersects the union of
		all the exceptional divisors in $\bl_j(\Jac(C))$ transversely, where $\bl_j(C_{2k})$ coincides with the blowup space associated to the chain $\{\delta\}_{j=0}^{k-1}$.
	\item[(b)] There is a natural embedding $\Jac(C)-\delta_{n}(C_{2n})\hookrightarrow \bl_{n+1}(\Jac(C))$, whose complement has $n+1$ smooth components with normal crossings.
\end{itemize}

Finally, there is a natural identification
\[ \bl_j(C_{2j}\times \bP^{k-j})=\bl_j(C_{2j})\times \bP^{k-j},\]
which induces a Cartesian diagram for $j<k$
	\[\begin{tikzcd} \bl_j(C_{2j})\times \bP^{k-j}\arrow[d,hook,"\bl_j(\gamma_{j,k})"] \arrow[r,twoheadrightarrow,"p_1"] & \bl_j(C_{2j})\arrow[d,hook,"\bl_j(\delta_j)"] \\
	\bl_j(C_{2k}) \arrow[r, "\bl_j(\delta_k)"] &\bl_j(\Jac(C)) \end{tikzcd}\]

\end{prop}

Grant Proposition \ref{prop: chain of Abel Jacobi maps} for now, we can easily deduce the main theorem for hyperelliptic
curves of odd genus $g = 2n+1$.

\begin{proof}[Proof of Theorem \ref{theorem: log resolution of hyperelliptic theta divisors}]
	For the sake of clarity, let us denote by $\{\bl_i'(\Jac(C))\}_{i=1}^n$ the
	sequence of blowups described in the introduction, where at the $i$-th stage to get $
	\bl_i'(\Jac(C))$, we
	blow up the strict transform of the Brill-Noether variety $W_{g-1}^{n+1-i}(C)$. We
	are going to argue that, in fact, 
	\[\bl_i'(\Jac(C)) = \bl_i(\Jac(C)).\] 
	
First, since $g=2n+1$, the image of $\bl_{n}(C_{2n})$ in $\bl_{n}(\Jac(C))$ is a
divisor, and so $\bl_{n+1}(\Jac(C))=\bl_n(\Jac(C))$; therefore both sequences really
have only $n$ steps. Note that $\delta_{i}(C_{2i})=W^{n-i}_{g-1}(C)$, which is reduced by Proposition \ref{prop: reducedness of Wrd}.  Proposition \ref{prop: chain of Abel Jacobi maps} gives an embedding 
\[\bl_i(C_{2i})\hookrightarrow
\bl_i(\Jac(C)).\]
By induction on $i$, it follows
easily that $\bl_i'(\Jac(C))=\bl_{i}(\Jac(C))$, and that the proper transform of
$W^{n-i}_{g-1}(C)$ in $\bl_{i}(\Jac(C))$ is equal to the image of $\bl_{i}(C_{2i})\hookrightarrow \bl_{i}(\Jac(C))$,
hence smooth.

The conclusion is that $\bl_n'(\Jac(C)) = \bl_n(\Jac(C))$ is smooth, and that the
strict transform $\tilde{\Theta}$ is the image of
$\bl_{n}(C_{2n})\hookrightarrow\bl_{n}(\Jac(C))$, hence also smooth. Since
$\{\delta_i\}_{i=0}^n$ is an NCD chain, the pullback $\pi_n^{\ast}\Theta$ is a
divisor with simple normal crossings. The multiplicity of the exceptional divisor
$Z_i$ equals the multiplicity of $\Theta$ at a point in
$W^{n-i}_{g-1}-W^{n+1-i}_{g-1}$, which is $n+1-i$ by the Riemann Singularity Theorem. 
\end{proof}

\begin{remark}\label{remark: why Xij is not a fiber product}
We show $\{\delta_j\}_{j=0}^n$ is a proper chain using the diagram 
\[ \begin{tikzcd}
C_{2i} \times \bP^{j-i} \arrow[r,twoheadrightarrow,"p_1"] \arrow[d,"\gamma_{i,j}"] & C_{2i}\arrow[d,"\delta_{i}"]\\
C_{2j} \arrow[r,"\delta_j"] & \Jac(C).
\end{tikzcd}
\]
However, it is easy to check that this diagram is not Cartesian and we will prove later that it is only Cartesian over suitable open subsets (see Lemma \ref{lemma: intersection}(a)). On the other hand, Proposition \ref{prop: chain of Abel Jacobi maps} shows that it becomes Cartesian after sufficiently many blow-ups.

\end{remark}

\section{Secant bundles and maps between them}\label{sec:Secant bundles and maps between them}

The proofs of Propositions \ref{prop: chain of addition maps} and \ref{prop: chain of Abel Jacobi maps}  
 rely on the geometry of secant bundles over
symmetric products of curves. In this section, we review the necessary definitions
and results, following the notation in Betram's work \cite{Bertram92}. 

Let $C$ be a smooth projective curve of genus $g \geq 0$, let $M$ be a line bundle on
$C$, and let $j\geq 0$ be an integer. We denote by $C_j = \Sym^j C$ the $j$-th
symmetric product of the curve. Consider the following diagram:
\[ \begin{tikzcd}
\sD_{j+1} \arrow[r,hook] & C\times C_{j+1} \arrow[dl,"p_1"] \arrow[dr,"p_2"]\\
C &  & C_{j+1}
\end{tikzcd}
\]
Here $\sD_{j+1}=C\times C_{j}$ is the universal divisor of degree $j+1$ over $C_{j+1}$, embedded via $(p,D)\mapsto (p,p+D)$. We say that $M$ separates $d$ points if
\[ h^0(C,M)=h^0(C,M(-D))+d, \quad \forall D\in C_{d}.\]
If $M$ separates $j+1$ points, then the sequence of sheaves 
\begin{equation}\label{eqn: ses for secant bundles}
0 \to p_1^{\ast}M\otimes \cO(-\sD_{j+1}) \to p_1^{\ast}M \to p_1^{\ast}M\otimes \cO_{\sD_{j+1}} \to 0
\end{equation}
on $C \times C_{j+1}$ remains exact when pushed down to $C_{j+1}$.
\begin{definition}\label{definition: secant bundles}
The secant bundle of $j$-planes over $C_{j+1}$, with respect to $M$, is 
\[ B^{j}(M)\colonequals \bP (p_2)_{\ast}(p_1^{\ast}M\otimes \cO_{\sD_{j+1}}).\]
\end{definition}
This is a $\bP^j$-bundle over the symmetric product $C_{j+1}$; for $j=0$, we have
$B^0(M) = C$. If $M$ separates $j+1$ points, the natural map to $\bP H^0(C,M)$ is
\begin{equation}\label{eqn: beta}
	\beta_j: B^j(M) \to \bP (p_2)_{\ast}(p_1^{\ast}M)=\bP H^0(C,M)\times C_{j+1} \to \bP H^0(C,M),
\end{equation}
where the last map is the projection to $\bP H^0(C,M)$. 

Assuming that $M$ separates $m+1$ points, we get a proper chain
\begin{equation}\label{eqn: the chain of secant varieties}
	\bigl\{\beta_j: B^j(M) \to \bP H^0(C,M)\bigr\}_{j=0}^{m},
\end{equation}
using the following diagram
\[ \begin{tikzcd}
B^i(M)\times C_{j-i} \arrow[r,twoheadrightarrow,"p_1"] \arrow[d,"\alpha_{i,j}"] &B^i(M)\arrow[d,"\beta_i"]\\
B^j(M)  \arrow[r,"\beta_j"] & \bP H^0(C,M)
\end{tikzcd}\]
Here $p_1$ is again the projection to the first coordinate. For $i<j$, the map 
\begin{equation}\label{eqn: alpha_m,k}
    \alpha_{i,j}:B^i(M)\times C_{j-i} \to B^j(M)
\end{equation}
is induced by the addition map $r:C_{i+1}\times C_{j-i} \to C_{j+1}$ (see the second definition in \cite[Page 432]{Bertram92}). This chain is proper because the image $\beta_j( B^j(M))$ is exactly the usual secant variety $\mathrm{Sec}^j(C)$ of $j$-planes through $j+1$ points of $C$ inside $\bP H^0(C,M)$ \cite[Page 432]{Bertram92}.

In order to study this chain, Betram introduced certain auxiliary chains, just as in the
previous section. Fix $k<m$, one can show that the chain 
\[\{\alpha_{i,k}:B^{i}(M)\times C_{k-i}\to B^k(M)\}_{i=0}^{k-1}\] is a proper chain using the diagram
\[ 
  \begin{tikzcd}
    B^i(M)\times C_{j-i}\times C_{k-j} \arrow
	 [r,twoheadrightarrow,"\mathrm{id}\times r"] \arrow[d,"\alpha_{i,j}\times \id"] &B^i(M)\times C_{k-i} \arrow[d,"\alpha_{i,k}"]\\
    B^j(M)\times C_{k-j} \arrow[r,"\alpha_{j,k}"] & B^k(M) 
    \end{tikzcd}
\]
where $r:C_{j-i}\times C_{k-j}\to C_{k-i}$ is the addition map on symmetric products. The chain $\{\alpha_{i,k}\}_{i=0}^{k-1}$ is proper because the image $\alpha_{i,k}(B^i(M)\times C_{k-i})$ can be thought as the relative secant variety of $i$-planes in $B^k(M)$, see \cite[Page 433]{Bertram92}.

Lastly, using Notation \ref{notation: chain product with a
constant factor}, we have for each pair $(j,k)$ with $j < k$ a proper chain
\[    
	\bigl\{\alpha_{i,j} \times \id:(B^i(M)\times C_{j-i})\times C_{k-j}\to
	B^j(M)\times C_{k-j}\bigr\}_{i=0}^{j-1}.
\]
In \cite[Proposition 2.2, Proposition 2.3]{Bertram92}, Bertram proved the following
result.
\begin{prop}[Bertram]\label{prop: injective maps of chains for secant bundles}
Let $M$ be a line bundle on $C$. 
\begin{enumerate}
	\item [(a)] If $M$ separates $m+1$ points, then for $j<k<m$, both $\{\alpha_{i,k}\}_{i=0}^{k-1}$ and $\{\alpha_{i,j} \times
		\id\}_{i=0}^{j-1}$ are chains of smooth centers, and the map 
		\[
			\alpha_{j,k}: B^j(M)\times C_{k-j}\to B^k(M)
		\]
		is an injective map of chains from $\{\alpha_{i,j} \times \id\}_{i=0}^{j-1}$ to  $\{\alpha_{i,k}\}_{i=0}^{j-1}$.
	\item [(b)] If $M$ separates $2k+2$ points, then $\{\beta_j\}_{j=0}^k$ is a chain
		of smooth centers, and 
		\[
			\beta_k:B^k(M) \to \bP H^0(C,M)
		\]
		is an injective map of chains from $\{\alpha_{j,k}\}_{j=0}^{k-1}$ to $\{\beta_j\}_{j=0}^{k-1}$.
\end{enumerate}
\end{prop}
\begin{remark}
Bertram's proof actually shows that $\{\alpha_{j,k}\}_{j=0}^{k-1}$ and $\{\beta_{j}\}_{j=0}^{k}$ are NCD chains. But these facts will not be used later.
\end{remark}

Let us spell out in detail what Bertram's theorem says in the case of $\bP^1$, where
the images of the secant bundles for $\cO_{\bP^1}(d)$ are the secant
varieties to the rational normal curve of degree $d$ in $\bP^d$.

\begin{corollary}
	\label{cor: injective maps of chains for secant bundles over rational curves}    
	Let $d \geq 2k+1$ and consider the line bundle $M=\cO_{\bP^1}(d)$ on $\bP^1$. Then $\{\beta_j\}_{j=0}^k$ is a chain of smooth centers. Fix $0 \leq i < j < k$, then 
\begin{enumerate}
   \item [(a)] The diagram
	\[ \begin{tikzcd}
		\bl_i \bigl( B^i(M)\times \bP^{j-i} \bigr) \times \bP^{k-j}
		\dar[hook]{\bl_i(\alpha_{i,j}) \times \id} \arrow[r,twoheadrightarrow] & 
		\bl_i \bigl( B^i(M) \times \bP^{k-i} \bigr)  \dar[hook]{\bl_i(\alpha_{i,k})} \\
			 \bl_i B^j(M) \times \bP^{k-j}  \rar{\bl_i(\alpha_{j,k})} & \bl_i B^k(M)
 	\end{tikzcd} \]
	is Cartesian and the two vertical arrows are embeddings. In particular, $\alpha_{j,k}$ is a map of chains of centers.
   \item [(b)] The diagram
	\[ \begin{tikzcd}
			\bl_i \bigl( B^i(M) \times \bP^{j-i} \bigr) 
			\dar[hook]{\bl_i(\alpha_{i,j})}  \arrow[r,twoheadrightarrow] & 
		\bl_i B^i(M) \dar[hook]{\bl_i(\beta_i)} \\
		\bl_i B^j(M) \rar{\bl_i(\beta_j)} & \bl_i \bP^d  	\end{tikzcd} \]
	is Cartesian and the two vertical arrows are embeddings. In particular, $\beta_j$ is a map of chains of centers.
	\item [(c)] There is a natural isomorphism
	\[ \bl_i(B^i(M)\times \bP^{j-i})\cong \bl_i B^i(M)\times \bP^{j-i}.\]
\end{enumerate}
\end{corollary}

\begin{proof}
Since $M=\cO_{\bP^1}(d)$ separates $d+1$ points on $\bP^1$, we can apply Proposition
\ref{prop: injective maps of chains for secant bundles} and use the isomorphisms
$\bP^k \cong \Sym^k\bP^1$ and $\bP^d \cong \bP H^0(\bP^1, M)$. 
The last statement follows from Lemma \ref{lemma: chains with a constant factor}.
\end{proof}

\section{Properties of Abel-Jacobi maps and addition maps}\label{sec: calculation of conormal bundles}
In this section, let $C$ be a hyperelliptic curve of genus $g=2n+1$. As a preparation for the proofs of Proposition \ref{prop: chain of
addition maps} and \ref{prop: chain of Abel Jacobi maps}, we establish some basic properties of the map $\gamma_{i,j}:C_{2i}\times \bP^{j-i} \to C_{2j}$ from \eqref{eqn: the chain of maps to symmetric products}
and of the Abel-Jacobi map $\delta_j: C_{2j} \to \Jac(C)$. In particular, their
conormal bundles are calculated in terms of the secant bundles over symmetric
products of $\bP^1$. In fact, it is known that the conormal bundle of the Abel-Jacobi
map of any curve can be described in terms of Steiner bundles (see \cite[Theorem 1.1]{Ein91}). For our purpose, it is more natural to use secant bundles.

We start by collecting some basic facts about hyperelliptic curves. From \cite[Page 13]{ACGH}, we know for each divisor $D$ with $h^0(\cO_C(D))=r+1$ and degree $d\leq g(C)$, there is a unique decomposition
\begin{equation*}
    D=E+\sum_{\ell=1}^{r}(p_\ell+q_\ell),
\end{equation*}
such that $p_\ell+q_\ell$ are hyperelliptic pairs and $E$ is a degree $d-2r$ divisor with $h^0(\cO_C(E))=1$. Similarly, for any
$L\in \Jac(C)$ with $h^0(L)=r+1$, there is a unique decomposition
\begin{equation}\label{eqn: decomposition of line bundles}
    L=rg^1_2\otimes L'
\end{equation}
such that $h^0(L')=1$.
\begin{notation}\label{notation: open part of symmetric product U2j}
	For each $j$, we define $U_{2j} \colonequals C_{2j}-\gamma_{j-1,j}(C_{2j-2}\times \bP^1)$. In other words, $U_{2j}$ consists of divisors $D$ where none of the degree 2 subdivisors of $D$ form a hyperelliptic pair and $h^0(\cO_C(D))=1$.
\end{notation}
\begin{notation}\label{notation: associated divisors on P1}
By definition, any $D\in U_{2j}$ gives a degree $2j$
divisor on $\bP^1$ by pushforward along the hyperelliptic map $h : C\to \bP^1$. We denote this divisor on
$\bP^1$ by the symbol $h_{\ast} D$ and define
\[\cO_{\bP^1}(g-1-h_{\ast} D)\colonequals \cO_{\bP^1}(g-1)\otimes
\cO_{\bP^1}(-h_{\ast} D),\]
which is a line bundle of degree $g-1-2j$.
\end{notation} 

Since $h:C\to \bP^1$ has $g+1$ branch points, one has
$h_{\ast}\cO_C \cong \cO_{\bP^1}(-1-g)\oplus \cO_{\bP^1}$. As we have $\omega_C \cong
h^{\ast} \cO_{\bP^1}(g-1)$, the projection formula gives
\[
	h_{\ast} \omega_C \cong \omega_{\bP^1} \oplus \cO_{\bP^1}(g-1).
\]
Consider the short exact sequence
$0 \to \cO_C(-D) \to
\cO_C \to \cO_D \to 0$ for $D\in U_{2j}$. Because $h$ is finite, pushing this forward along $h$ gives
\[ 0 \to h_{\ast}\cO_C(-D) \to
h_{\ast}\cO_C \to h_{\ast}\cO_D \to 0.\]
Since $h$ is an isomorphism over $D$, we have $h_{\ast}\cO_D=\cO_{h_{\ast}D}$.
Moreover, as $h_{\ast}D$ is supported away from the branch points of $h$, the map $h_{\ast}\cO_C\to h_{\ast}\cO_D$ is the composition of
\[ h_{\ast}\cO_C\to \cO_{\bP^1}\to \cO_{h_{\ast}D}=h_{\ast}\cO_D,\]
where the first map is the projection induced by $h_{\ast}\cO_C=\cO_{\bP^1}(-1-g)\oplus \cO_{\bP^1}$.  It follows that $h_{\ast}\cO_C(-D)=\cO_{\bP^1}(-1-g)\oplus \cO_{\bP^1}(-h_{\ast}D)$. Consequently, using projection formula again gives
\begin{equation} \label{eqn: canonical bundle and xi_D}
	h_{\ast} \omega_C(-D) \cong \omega_{\bP^1} \oplus \cO_{\bP^1}(g-1-h_{\ast} D).
\end{equation}

In the following statement, see \eqref{eqn: notation of conormal bundle} for the notation of conormal bundle $N^{\ast}_f$ or $N^{\ast}_{X|Y}$.
\begin{lemma}\label{lemma: conormal of addition maps}
For $0\leq i<j\leq n$ and $\gamma_{i,j}$ from \eqref{eqn: the chain of maps to symmetric products}, we have
\begin{enumerate}
	\item [(a)] $d\gamma_{i,j}: \gamma_{i,j}^{\ast} \ctbl{C_{2j}} \to
		\ctbl{C_{2i}\times \bP^{j-i}}$ is surjective when restricted to $U_{2i}\times \bP^{j-i}$. 
    \item [(b)] For $i=0$, we have an isomorphism
    \[ \bP N^{\ast}_{\bP^j|C_{2j}} \cong B^{j-1}(\cO_{\bP^1}(g-1)),\]
    the latter is the secant bundle over $\Sym^j\bP^1 = \bP^j$ with respect to $\cO_{\bP^1}(g-1)$.
    \item [(c)] For $i\geq 1$, the space $\bP N^{\ast}_{\gamma_{i,j}} \big\vert_{U_{2i}\times \bP^{j-i}}$ is smooth over $U_{2i}$, such that over $D\in U_{2i}$ we have an isomorphism:
		 \begin{equation}\label{eqn: identification between conormal and secant bundle}
		 \bP N^{\ast}_{\gamma_{i,j}} \big\vert_{\{D\}\times \bP^{j-i}} \cong
	 B^{j-i-1}(\cO_{\bP^1}(g-1-h_{\ast} D)), \end{equation}
    the secant bundle over $\bP^{j-i}$ with respect to
	 $\cO_{\bP^1}(g-1-h_{\ast} D)$. 
\end{enumerate}    
\end{lemma}

\begin{proof}
As a warm-up, let us calculate the conormal bundle of $\bP^j$ inside $C_{2j}$.
Recall that for any divisor $D\in C_{2j}$, there is a canonical identification \cite[Page 160]{ACGH}
\[ \ctbl{C_{2j}} \big\vert_D \cong H^0(C, \omega_C \otimes \cO_D).\]
Under the isomorphism $\bP^j \cong \Sym^j \bP^1$, the morphism $\bP^j \to C_{2j}$
associates an effective divisor $E$ of degree $j$ on $\bP^1$ to an effective divisor
$h^{\ast} E$ of degree $2j$ on $C$ and we have also
\[
	\ctbl{\bP^j} \big\vert_E \cong H^0(\bP^1, \omega_{\bP^1} \otimes \cO_E).
\]
 Moreover,  $\ctbl{C_{2j}} \big\vert_{h^{\ast}E}$ and $\ctbl{\bP^j} \big\vert_E$ can be related in the following way. One has
\begin{align*}
	\ctbl{C_{2j}} \big\vert_{h^{\ast} E} &\cong H^0(C, \omega_C \otimes h^{\ast} \cO_E)\\
	&\cong H^0(\bP^1, h_{\ast} \omega_C \otimes \cO_E) \\
		&\cong H^0(\bP^1, \omega_{\bP^1} \otimes \cO_E) \oplus
	H^0(\bP^1, \cO_{\bP^1}(g-1) \otimes \cO_E)\\
	&=\ctbl{\bP^j} \big\vert_E \oplus
	H^0(\bP^1, \cO_{\bP^1}(g-1) \otimes \cO_E).
\end{align*}
It follows that the map $\ctbl{C_{2j}} \big\vert_{h^{\ast} E} \to \ctbl{\bP^j} \big\vert_E$ induced by $\bP^j\hookrightarrow C_{2j}$ is surjective with
kernel
\[
	N_{\bP^j|C_{2j}}^{\ast} \big\vert_E \cong H^0(\bP^1, \cO_{\bP^1}(g-1) \otimes
	\cO_E).
\]
Varying this isomorphism over $E\in \bP^j$ gives
\[
	N_{\bP^j|C_{2j}}^{\ast} \cong (p_2)_{\ast} \bigl( p_1^{\ast} \cO_{\bP^1}(g-1) \otimes
	\cO_{\sE_j} \bigr),
\]
where $\sE_j$ denotes the universal divisor over $\bP^j=\Sym^j\bP^1$ as follows
\begin{equation}\label{eqn: diagram for symmetric product over P1}
    \begin{tikzcd}
\sE_{j} \arrow[r,hook] & \bP^1\times \Sym^j\bP^1 \arrow[dl,"p_1"] \arrow[dr,"p_2"]\\
\bP^1 &  & \Sym^j\bP^1
\end{tikzcd}
\end{equation} 
In particular, the right hand side is the secant bundle
$B^{j-1} \bigl( \cO_{\bP^1}(g-1) \bigr)$ and it proves (b).

For (a) and (c), consider $D \in U_{2i}$ and $E \in \bP^{j-i}$. The morphism
$\gamma_{i,j} : U_{2i} \times \bP^{j-i} \to C_{2j}$ takes the pair $(D,E)$ to the
divisor $D + h^{\ast} E$ of degree $2j$ on $C$.  Because $D \in U_{2i}$, we have $H^0(C, \cO_C(D+h^{\ast} E)) = H^0(C, \cO_C(h^{\ast}
E))$ by \eqref{eqn: decomposition of line bundles}. After a little bit of diagram chasing, this gives us a diagram involving a short exact sequence
\[ \begin{tikzcd}
	0 \arrow[r] &H^0(C, \omega_C(-D) \otimes \cO_{h^{\ast} E}) \arrow[r] &
	H^0(C, \omega_C \otimes \cO_{D + h^{\ast} E}) \arrow[d,equal] \arrow[r]&
	H^0(C, \omega_C \otimes \cO_D) \arrow[r] \arrow[d,equal]& 0\\
 {} & {} & \ctbl{C_{2j}} \big\vert_{D + h^{\ast} E} \arrow[r] & \ctbl{U_{2i}} \big\vert_D & 
\end{tikzcd}\]
where the bottom map is induced by the cotangent map of
\begin{align*}
U_{2i} &\hookrightarrow U_{2i}\times \bP^{j-i}\to C_{2j},\\
D&\mapsto (D,E) \mapsto \gamma_{i,j}(D,E).
\end{align*}
Using \eqref{eqn: canonical bundle and xi_D}, one has
\begin{align*}
	\Ker \Bigl( \ctbl{C_{2j}} \big\vert_{D + h^{\ast} E} \to
	\ctbl{U_{2i}} \big\vert_D  \Bigr) 
	&\cong H^0(C, \omega_C(-D) \otimes h^{\ast} \cO_E)\\
	&\cong H^0(\bP^1, h_{\ast} \omega_C(-D) \otimes \cO_E)\\
	&\cong H^0(\bP^1, \omega_{\bP^1} \otimes \cO_E) \oplus
	H^0(\bP^1, \cO_{\bP^1}(g-1-h_{\ast} D) \otimes \cO_E)\\
	&=\ctbl{\bP^{j-i}} \big\vert_E\oplus
	H^0(\bP^1, \cO_{\bP^1}(g-1-h_{\ast} D) \otimes \cO_E),
\end{align*}
and the morphism to the cotangent space of $\bP^{j-i}$ is the projection to the
first summand. Hence, we deduce that
\[
	d\gamma_{i,j} : \ctbl{C_{2j}} \big\vert_{D + h^{\ast} E} \to 
	\ctbl{U_{2i}} \big\vert_D \oplus \ctbl{\bP^{j-i}} \big\vert_E
\]
is surjective, proving (a); and that its kernel can be identified with
\begin{equation}\label{eqn: conormal bundle of N restricting to D,E}
	N_{\gamma_{i,j}}^{\ast} \restr{(D, E)} \cong 
	H^0(\bP^1, \cO_{\bP^1}(g-1-h_{\ast} D) \otimes \cO_E).
\end{equation}
This isomorphism is natural in $E\in \bP^{j-i}$, and therefore
\[
	N_{\gamma_{i,j}}^{\ast} \big\vert_{ \{D\} \times \bP^{j-i}} 
	\cong (p_2)_{\ast} \bigl( p_1^{\ast} \cO_{\bP^1}(g-1-h_{\ast} D) 
	\otimes \cO_{\sE_{j-i}} \bigr).
\]
The latter is a vector bundle on $\bP^{j-i}$ because the line bundle
$\cO_{\bP^1}(g-1-h_{\ast} D)$ separates $j-i$ points (on account of the inequality
$g-1-2i =2n-2i> j-i$). Varying $D\in U_{2i}$, we see that the
projectivized conormal bundle $N^{\ast}_{\gamma_{i,j}}|_{U_{2i}\times \bP^{j-i}}$ is a projective bundle over
$U_{2i}$, hence is smooth over $U_{2i}$. Moreover, its fiber over $D \in U_{2i}$ is the secant bundle
$B^{j-i-1}(\cO_{\bP^1}(g-1-h_{\ast} D))$, which proves (c).
\end{proof}

\begin{remark}
This lemma is parallel to \cite[Lemma 1.3]{Bertram92}, with the difference that the
relevant divisor is $h_{\ast} D$, not $2 h_{\ast} D$ as in \cite{Bertram92}.
\end{remark}
From the proof of Lemma \ref{lemma: conormal of addition maps}, we can deduce one additional useful fact. For $0
\leq i<j<k \leq n$, consider the commutative diagram induced by \eqref{eqn: diagram
for addition maps} via restriction:
\[ \begin{tikzcd}
(U_{2i}\times \bP^{j-i}) \times \bP^{k-j} \arrow[d,hook,"\gamma_{i,j}\times \id"]
\arrow[r,twoheadrightarrow,"\id \times r"] & U_{2i}\times \bP^{k-i}  \arrow[d,hook,"\gamma_{i,k}"]\\
C_{2j}\times \bP^{k-j} \arrow[r,"\gamma_{j,k}"] & C_{2k}
\end{tikzcd}
\]

\begin{corollary}\label{corollary: conormal of addition maps}
For $D\in U_{2i}$, the induced map of conormal bundles 
\[ \epsilon:(\id\times r)^{\ast}N^{\ast}_{\gamma_{i,k}} \big\vert_{\{D\}\times
	 \bP^{k-i}} \to N^{\ast}_{\gamma_{i,j}\times \id} \big\vert_{\{D\}\times \bP^{j-i}\times \bP^{k-j}}\]
	 on $\{D\}\times \bP^{j-i}\times \bP^{k-j}$ is surjective, and the diagram
	 \[ \begin{tikzcd}[column sep=huge] 
			 \bP N^{\ast}_{\gamma_{i,j}\times \id} \big\vert_{\{D\}\times
	 \bP^{j-i}\times \bP^{k-j}} \arrow[r,"\alpha"] \arrow[d,"\eqref{eqn: identification between conormal and secant bundle}\times \mathrm{id}_{\bP^{k-j}}"] & \bP
	 N^{\ast}_{\gamma_{i,k}} \big\vert_{\{D\}\times \bP^{k-i}} \arrow[d,"\eqref{eqn: identification between conormal and secant bundle}"] \\
 B^{j-i-1}(M) \times \bP^{k-j} \arrow[r,"\alpha_{j-i-1,k-i-1}"] & B^{k-i-1}(M)
    \end{tikzcd}\]
	 commutes.  Here $\alpha$ is induced by $\epsilon$ and the projection to $\bP
 N^{\ast}_{\gamma_{i,k}} \big\vert_{\{D\}\times \bP^{k-i}}$, both vertical arrows are
 isomorphisms, and $\alpha_{j-i-1,k-i-1}$ is the map in \eqref{eqn: alpha_m,k} for
 the curve $\bP^1$ and the line bundle $M = \cO_{\bP^1}(g-1-h_{\ast} D)$. 
 
\end{corollary}

\begin{proof}
	To simplify the notation, fix a point $D \in U_{2i}$ and denote $M =
	\cO_{\bP^1}(g-1-h_{\ast} D)$. For $E_1 \in
	\bP^{j-i}$ and $E_2 \in \bP^{k-j}$, according to \eqref{eqn: conormal bundle of N restricting to D,E} the map
\[
	\epsilon \vert_{(D, E_1, E_2)} : N_{\gamma_{i,k}}^{\ast} \big\vert_{(D, E_1 + E_2)}
	\to N_{\gamma_{i,j}}^{\ast} \big\vert_{(D, E_1)}
\]
is identified with the map
\[
	H^0 \bigl( \bP^1, M \otimes \cO_{E_1 + E_2} \bigr) \to
	H^0 \big( \bP^1, M \otimes \cO_{E_1} \bigr),
\]
which is obviously surjective. The remaining assertion is clear from \eqref{eqn:
alpha_m,k}.
\end{proof}

\begin{lemma}\label{lemma: conormal of abel jacobi maps}
For $0\leq j\leq n$, consider the map $\delta_j$ from \eqref{eqn: the chain of maps to Jacobian}, then the following holds.
\begin{enumerate}
	\item [(a)]$d\delta_{j}:\delta_j^{\ast} \ctbl{\Jac(C)} \to \ctbl{C_{2j}}$ is surjective when restricted to $U_{2j}$. 
    \item [(b)] The fiber of $N^{\ast}_{\delta_{j}}$ over
		 $D\in U_{2j}$ is $ H^0(\bP^1,\cO_{\bP^1}(g-1- h_{\ast} D))$.
\end{enumerate}    
\end{lemma}

\begin{proof}
We only sketch the proof, as it is similar to that of Lemma \ref{lemma: conormal of
addition maps}. The cotangent bundle $\ctbl{\Jac(C)}$ is a trivial bundle with fibers
isomorphic to $H^0(C, \omega_C)$, and the map $d\delta_{j}$ over $D\in C_{2j}$ can be identified with
\begin{equation}\label{eqn: tangent map of Abel Jacobi}
	H^0(C, \omega_C) \to H^0(C, \omega_C \otimes \cO_D).
\end{equation}
The cokernel of this map is the kernel of $H^1(\omega_C(-D))\to H^1(\omega_C)$; dually one has the natural map $H^0(\cO_C)\to H^0(\cO_C(D))$. If $D \in U_{2j}$, then $h^0(\cO_C(D))=1$ and the last map is an isomorphism. Therefore the map \eqref{eqn: tangent map of Abel Jacobi} is surjective and we have
\begin{equation}\label{eqn: kernel of conormal of delta j}
	N_{\delta_j}^{\ast} \big\vert_D \cong H^0 \bigl( C, \omega_C(-D) \bigr)
	\cong H^0 \bigl( \bP^1, \cO_{\bP^1}(g-1-h_{\ast} D) \bigr),
\end{equation}
using \eqref{eqn: canonical bundle and xi_D}.
\end{proof}

Let us record one additional useful fact. Consider the
commutative diagram induced by \eqref{eqn: diagram for abel jacobi maps} via restriction $(i<j)$:
\[ \begin{tikzcd}
U_{2i} \times \bP^{j-i} \arrow[d,hook,"\gamma_{i,j}"] \arrow[r,twoheadrightarrow,"p_1"] & U_{2i} \arrow[d,hook,"\delta_{i}"]\\
C_{2j} \arrow[r,"\delta_{j}"] & \Jac(C)
\end{tikzcd}
\]

\begin{corollary}\label{corollary: conormal of abel jacobi maps}
Fix $0 \leq i < j \leq n$. For $D\in U_{2i}$, the induced map of conormal bundles 
\[ \epsilon: p_1^{\ast} N^{\ast}_{\delta_{i}} \big\vert_D \to N^{\ast}_{\gamma_{i,j}}
\big\vert_{\{D\}\times \bP^{j-i}}\]
over $\{D\} \times \bP^{j-i}$ is surjective, and the diagram 
   \[\begin{tikzcd}
			\bP N^{\ast}_{\gamma_{i,j}}\big\vert_{\{D\}\times \bP^{j-i}}
			\arrow[r,"\beta"] \arrow[d,"\eqref{eqn: identification between conormal and secant bundle}"]& \bP N^{\ast}_{\delta_{i}} \big\vert_D \arrow[d,"\cong"] \\
			B^{j-i-1}(M) \arrow[r,"\beta_{j-i-1}"] & \bP H^0(\bP^1,M)
   \end{tikzcd} 
   \]
	commutes. Here the first vertical map is an isomorphism by Lemma \ref{lemma: conormal
	of addition maps}(c), the second isomorphism comes from Lemma \ref{lemma: conormal of abel jacobi maps},  $\beta_{j-i-1}$ is the map \eqref{eqn: beta} for the curve
	$\bP^1$ and the line bundle $M = \cO_{\bP^1}(g-1-h_{\ast} D)$, and $\beta$ is
	induced by $\epsilon$ and the projection to $\bP N^{\ast}_{\delta_{i}} \big\vert_D$.
\end{corollary}

\begin{proof}
	Fix a point $D \in U_{2i}$ and define $M =
	\cO_{\bP^1}(g-1-h_{\ast} D)$. For $E \in
	\bP^{j-i}$, by \eqref{eqn: conormal bundle of N restricting to D,E} and \eqref{eqn: kernel of conormal of delta j}, the map 
\[
	\epsilon \vert_{(D, E)} : N_{\delta_i}^{\ast} \big\vert_D
	\to N_{\gamma_{i,j}}^{\ast} \big\vert_{(D, E)}
\]
between the fibers of the two conormal bundles is identified with the map
\[
	H^0(\bP^1, M) \to H^0 \big( \bP^1, M \otimes \cO_E \bigr),
\]
which is surjective by degree reasons ($\deg M=g-1-2i>j-i=\deg E$). The remaining assertion is clear from
\eqref{eqn: beta}.
\end{proof}

We end this section by some properties of $\delta_j$ and $\gamma_{i,j}$ restricting to suitable open subsets.
\begin{lemma}\label{lemma: intersection}
Fix $0\leq i<j \leq n$.
\begin{enumerate}
    \item [(a)] The maps
\begin{align*}
\delta_{j}:U_{2j} \to \Jac(C), \quad \gamma_{i,j}:U_{2i}\times \bP^{j-i} \to C_{2j}
\end{align*}
are embeddings and the restriction of the diagram \eqref{eqn:
	 diagram for abel jacobi maps} is  Cartesian:
    \[\begin{tikzcd}
    U_{2i}\times \bP^{j-i} \arrow[r,twoheadrightarrow,"p_1"] \arrow[d,hook,"\gamma_{i,j}"] & U_{2i} \arrow[d,hook,"\delta_{i}"]\\ 
    C_{2j} \arrow[r,"\delta_{j}"] &\Jac(C)     \end{tikzcd}
    \]
Equivalently, we have
\[ \delta_j^{-1}(U_{2i})=U_{2i}\times \bP^{j-i}.\]    
    \item [(b)] For $0 \leq i < j < k \leq n$, the restriction of the diagram \eqref{eqn: diagram for addition maps} is Cartesian:
    \[ \begin{tikzcd}
    (U_{2i}\times \bP^{j-i}) \times \bP^{k-j} \arrow [r,"\id \times r"]
	 \arrow[d,hook,"\gamma_{i,j}\times \id"] & U_{2i} \times \bP^{k-i} \arrow[d,hook,"\gamma_{i,k}"]\\
    C_{2j}\times \bP^{k-j} \arrow[r,"\gamma_{j,k}"] & C_{2k} 
    \end{tikzcd}
    \]
    Equivalently, we have
\[ \gamma_{j,k}^{-1}(U_{2i}\times \bP^{k-i})=U_{2i}\times \bP^{j-i}\times \bP^{k-j}.\]  
    \item [(c)] In particular, for $j<k\leq n$ and $i=0$, we have $\gamma_{j,k}^{-1}(C_0\times \bP^k)=C_0\times \bP^{j}\times \bP^{k-j}$.
 
\end{enumerate}
\end{lemma}

\begin{proof}

As $\delta_i(D)=(n-i)g^1_2\otimes \cO_C(D)$ by \eqref{eqn: the chain of maps to Jacobian}, it follows immediately from the uniqueness of the decomposition \eqref{eqn: decomposition of line bundles} that
$\delta_i:U_{2i}\to \Jac(C)$ is injective, and that the image  $\delta_i(U_{2i})$ consists of line
bundles $L\in \Jac(C)$ such that $h^0(L)=n-i+1$.  Similarly, one can show that the map $\gamma_{i,j}:U_{2i}\times C_{j-i}\to C_{2j}$ is injective, and any divisor in its image can be written as
\[ D=E+\sum_{\ell=1}^{j-i}(p_\ell+q_\ell)\]
where $E$ is a degree $2i$ divisor with $h^0(\cO_C(E))=1$ and $p_\ell+q_\ell$ are hyperelliptic pairs. Therefore the image $\gamma_{i,j}(U_{2i}\times C_{j-i})$ consists of divisors $D$ of degree $2j$ such that
$h^0(\cO_C(D))=j-i+1$. To show the restricted maps $\delta_j$ and $\gamma_{i,j}$ are embeddings, one needs the surjectivity of
$d\delta_j$ and $d\gamma_{i,j}$, which follows from Lemma \ref{lemma: conormal of
addition maps}(a) and Lemma \ref{lemma: conormal of abel jacobi maps}(a).

Now we want to argue the diagrams in (a),(b) are Cartesian. We will identify $U_{2i}$ and $U_{2i}\times \bP^{j-i}$ with their images in $\Jac(C)$ and $C_{2j}$. We want to show
\[ \delta^{-1}_j(U_{2i})=U_{2i}\times \bP^{j-i}.\]
Let us first prove this set-theoretically. Suppose $D\in \delta^{-1}_j(U_{2i})$, this means that
\[ h^0(\delta_j(D))=h^0((n-j)g^1_2\otimes \cO_C(D))=n-i+1\]
by the characterization of the image of $U_{2i}$ in $\Jac(C)$. Using \eqref{eqn: decomposition of line bundles}, we must have 
\[ h^0(\cO_C(D))=(n-i+1)-(n-j)=j-i+1\]
and conclude that $D\in U_{2i}\times \bP^{j-i}$, by the characterization of its image in $C_{2j}$.  The argument for the set-theoretic part of (b),(c) is similar. To finish the proof of (a), we need to show the surjectivity of
\[ d\delta_j:\delta_j^{\ast}N^{\ast}_{U_{2i}|\Jac(C)} \to N^{\ast}_{U_{2i}\times
\bP^{j-i}|C_{2j}},\]
which follows from Corollary \ref{corollary: conormal of abel jacobi maps}. Similarly, the
statement that the diagram in (b) is Cartesian follows from Corollary \ref{corollary: conormal of addition maps}.
\end{proof}

\section{The proof of Proposition \ref{prop: chain of addition maps}}\label{sec: proof of chain of additions maps}

Let $C$ be a smooth hyperelliptic curve of genus $g = 2n+1$. Let $t\leq n$ be an integer and consider the proper chain  
$\{\gamma_{k,t}:C_{2k}\times \bP^{t-k} \to C_{2t}\}_{k=0}^{t-1}$ from \eqref{eqn: chain gamma ik} with the diagram \eqref{eqn: diagram for addition maps} for $(j<k<t)$
\[ \begin{tikzcd} 
C_{2j}\times \bP^{k-j} \times \bP^{t-k} \arrow[r,twoheadrightarrow,"\mathrm{id}\times r"] \arrow[d,"\gamma_{j,k}\times \mathrm{id}"]& C_{2j}\times \bP^{t-j} \arrow[d,"\gamma_{j,t}"] \\
C_{2k}\times\bP^{t-k}\arrow[r,"\gamma_{k,t}"] & C_{2t}
\end{tikzcd}\]
Denote by $\bl_j(C_{2k}\times \bP^{t-k}), \bl_j(C_{2t})$ the spaces associated to the chain $\{\gamma_{k,t}\}_{k=0}^{t-1}$ and denote by $\bl_i(C_{2j}\times \bP^{k-j}\times \bP^{t-k})$ the space associated to the chain $\{\gamma_{j,k}\times \mathrm{id}\}_{j=0}^{k-1}$.

Let $k$ be an integer with $0\leq k<t$. We break the proof of Proposition \ref{prop:
chain of addition maps} down into the following claims. (We apologize to the reader
for the fact that this looks so complicated; the problem is that the threefold
induction needs a lot of notation just to state everything correctly.)

\begin{claim}\label{claim: gammajt is NCD}
The chain $\{\gamma_{j,t}\}_{j=0}^{k}$ is a NCD chain.\end{claim}

\begin{claim}\label{claim: gammakt is a map of chains}
Let $j\leq k<t$, then
\begin{align}\label{eqn: gammakt is a map of chains}
\bl_j(\gamma_{k,t})^{-1}(\bl_j(C_{2j}\times \bP^{t-j}))&=\bl_j(C_{2j}\times \bP^{k-j}\times \bP^{t-k}).
\end{align}
In other words, $\gamma_{k,t}$ is a map of chains of centers.
\end{claim}

As a consequence, $\bl_j(C_{2k}\times \bP^{t-k})$ coincides with the space associated to the chain $\{\gamma_{j,k}\times \mathrm{id}\}_{j=0}^{k-1}$,  and by Lemma \ref{lemma: chains with a constant factor} there are natural isomorphisms
\begin{align*}
\bl_j(C_{2k}\times \bP^{t-k})&=\bl_j(C_{2k})\times \bP^{t-k},\\
\bl_i(C_{2j}\times \bP^{k-j}\times \bP^{t-k})&=\bl_i(C_{2j})\times \bP^{k-j}\times \bP^{t-k}, \quad \forall i<j.
\end{align*}

\begin{claim}\label{claim: pullback of Eij and complements}
Let $0\leq i<j\leq k<t$ and consider the following diagram 
\begin{equation*}
\begin{tikzcd}
{} & {} & \bl_j(C_{2j}\times \bP^{t-j})\arrow[d,hookrightarrow,"\bl_j(\gamma_{j,t})"] & \\
G_{i,j}\arrow[r,hookrightarrow] \arrow[d] &\bl_j(C_{2k}\times \bP^{t-k}) \arrow[r,"\bl_j(\gamma_{k,t})"] \arrow[d] &\bl_j(C_{2t}) \arrow[d]\arrow[r,hookleftarrow] &E_{i,j}\arrow[d]\\
\bl_i(C_{2i}\times \bP^{k-i}\times \bP^{t-k}) \arrow[r,hookrightarrow,"\bl_i(\gamma_{i,k}\times \id)"] &\bl_i(C_{2k}\times \bP^{t-k})\arrow[r,"\bl_i(\gamma_{k,t})"] &\bl_i(C_{2t}) \arrow[r,hookleftarrow,"\bl_i(\gamma_{i,t})"] &\bl_i(C_{2i}\times \bP^{t-i})
\end{tikzcd}
\end{equation*}
where $G_{i,j},E_{i,j}$ are the exceptional divisors over $\bl_i(C_{2i}\times  \bP^{k-i}\times \bP^{t-k})$ and $\bl_i(C_{2i}\times \bP^{t-i})$, respectively. Then
\begin{equation}\label{eqn: pullback of Eij is Gij}
\bl_{j}(\gamma_{k,t})^{-1}(E_{i,j})=G_{i,j},
\end{equation}
\begin{equation}\label{complement of bljC2j times Pt-j}
\bl_j(C_{2j}\times \bP^{t-j}) - \bigcup_{i<j} E_{i,j} = U_{2j}\times \bP^{t-j},
\end{equation}
where $U_{2j}=C_{2j}-\gamma_{j-1,j}(C_{2j-2}\times \bP^1)$.
\end{claim}

\begin{claim}\label{claim: blowup restricting to exceptional divisors for abel jacobi} 
Let $0\leq i<j\leq k<t$. Denote by $E_{i,j}^{\circ}=E_{i,j}-\cup_{h<i}E_{h,j}$ and the same for $G^{\circ}_{i,j}$. Then the map induced from \eqref{eqn: pullback of Eij is Gij}
\[\bl_{j}(\gamma_{k,t})\colon G^{\circ}_{i,j} \to E^{\circ}_{i,j}\]
 is a morphism of $U_{2i}$-varieties, whose fiber over $D\in U_{2i}$ is 
 \[\bl_{j-i-1}(\alpha_{k-i-1,t-i-1}):\bl_{j-i-1}(B^{k-i-1}(M)\times \bP^{t-k})\to \bl_{j-i-1}B^{t-i-1}(M),\] associated to the chain
\[\{\alpha_{\ell,t-i-1}:B^{\ell}(M)\times \bP^{t-i-1-\ell}\to B^{t-i-1}(M)\}_{\ell=0}^{k-i-1}\]
from \eqref{eqn: alpha_m,k} with respect to the  line bundle $M=\cO_{\bP^1}(g-1-h_{\ast} D)$ (see Notation \ref{notation: associated divisors on P1}).
\end{claim}

\begin{claim}\label{claim: Gij and Eij}
Let $0\leq i<j\leq k<t$, then
\begin{align*}
\bl_j(\gamma_{k,t})^{-1}(\bl_j(C_{2j}\times \bP^{t-j})-\bigcup_{i<j} E_{i,j})=\bl_j(C_{2j}\times \bP^{k-j}\times \bP^{t-k})-\bigcup_{i<j} G_{i,j},\tag{$\ast$}\\
\bl_j(\gamma_{k,t})^{-1}( \bl_j(C_{2j}\times\bP^{t-j}) \cap E_{i,j}^{\circ})=\bl_j(C_{2j}\times \bP^{k-j}\times \bP^{t-k})\cap G_{i,j}^{\circ}, \tag{$\ast\ast$}\\
\bl_k(\gamma_{k,t}):G_{i,k}^{\circ}\to E_{i,k}^{\circ} \textrm{ is an embedding}.\tag{$\ast\ast\ast$}
\end{align*} 
\end{claim}

We prove these claims by induction on the triples $(j,k,t)$, representing the
integers referenced in the statements of the claims. In what follows, when we say
something like, ``\emph{Claim \ref{claim: pullback of Eij and complements} holds for a triple
$(j,k,t)$}'', we mean that the statements in the claim hold for all $0<i< j\leq k<t$,
in the notation used in Claim \ref{claim: pullback of Eij and complements}.
This principle applies similarly to the other claims.

The base case consists of the following two cases:
\begin{enumerate}
    \item $k=0$ and arbitrary $t$. For Claim \ref{claim: gammajt is NCD}, it follows from the
fact that $\gamma_{0,t}:C_0\times \bP^t\hookrightarrow C_{2t}$ is an embedding of
smooth varieties. There is nothing to check for the other claims. 
\item $j=0$, arbitrary $k<t$ (except Claim \ref{claim: gammajt is NCD}). Lemma \ref{lemma: intersection}(c) says that
\[ \gamma_{k,t}^{-1}(C_0\times \bP^{t})=C_0\times \bP^{k}\times \bP^{t-k}.\]
Therefore Claim \ref{claim: gammakt is a map of chains} holds. The statements for
Claim \ref{claim: pullback of Eij and complements} to Claim \ref{claim: Gij and Eij} are empty. 
\end{enumerate} 
For the sake of clarity, let us restate the assumptions that we are allowed to use in
the induction in the following form:
\begin{assumption}
    Let $j,k,t$ be integers such that $j\leq k<t\leq n$. Suppose
    \begin{itemize}
        \item Claim \ref{claim: gammajt is NCD}  holds for all the tuples $(k',t')$ such that 
        \[ k'<t'\leq t-1, \quad \textrm{ or } \quad  k'\leq k-1, t'=t.\]
        \item Claim \ref{claim: gammakt is a map of chains} to Claim \ref{claim: Gij and Eij} hold for all the triples $(j',k',t')$ such that 
        \[j'\leq k'<t'\leq t-1, \quad \textrm{ or} \quad j'\leq k'\leq k-1, t'=t.\]
        \item Claim \ref{claim: gammakt is a map of chains} holds for all $(j',k,t)$ such that $j'\leq j-1$.
    \end{itemize}
\end{assumption}
The goal is to prove Claim \ref{claim: gammajt is NCD} holds for the tuple $(k,t)$
and Claim \ref{claim: gammakt is a map of chains} to Claim \ref{claim: Gij and Eij} hold for the triple $(j,k,t)$. 

\textbf{Step 1}: Claim \ref{claim: pullback of Eij and complements}. Let $i$ be an integer such that $i<j$. By inductive Claim \ref{claim: gammajt is NCD} and Claim \ref{claim: gammakt is a map of chains}, the diagram in Claim \ref{claim: pullback of Eij and complements} exists and the spaces $\bl_i(C_{2i}\times \bP^{t-i}),\bl_i(C_{2i}\times 
\bP^{k-i}\times \bP^{t-k})$ are smooth. The inductive \eqref{eqn: gammakt is a map of chains} implies that
\[\bl_i(\gamma_{k,t})^{-1}\left(\bl_i(C_{2i}\times \bP^{t-i})\right)=\bl_i(C_{2i}\times \bP^{k-i}\times \bP^{t-k}).\]
Since $G_{i,j},E_{i,j}$ are the corresponding exceptional divisors, a repetitive application of \eqref{eqn: functoriality of blow ups}  gives \eqref{eqn: pullback of Eij is Gij}.

To prove \eqref{complement of bljC2j times Pt-j}, consider the following diagram 
\[ \begin{tikzcd}
G_{i,j}' \arrow[r,hook] \arrow[d]& \bl_j(C_{2j}\times \bP^{t-j}) \arrow[r,hook,"\bl_j(\gamma_{j,t})"] \arrow[d]& \bl_j(C_{2t}) \arrow[d]\arrow[r,hookleftarrow]& E_{i,j}\arrow[d]\\
\bl_{i}(C_{2i}\times \bP^{j-i}\times \bP^{t-j}) \arrow[r,hook,"\bl_i(\gamma_{i,j}\times \mathrm{id})"] & \bl_i(C_{2j}\times \bP^{t-j}) \arrow[r,"\bl_i(\gamma_{j,t})"] & \bl_i(C_{2t}) \arrow[r,hookleftarrow,"\bl_i(\gamma_{i,t})"] &\bl_i(C_{2i}\times \bP^{t-i}),
\end{tikzcd}\]
where $G'_{i,j}$ is the exceptional divisor over $\bl_{i}(C_{2i}\times \bP^{j-i}\times \bP^{t-j})$. This diagram exists because $\{\gamma_{i,t}\}_{i=0}^{j-1}$ and $\{\gamma_{i,j}\times \mathrm{id}\}_{i=0}^{j-1}$ are NCD chains by the inductive Claim \ref{claim: gammajt is NCD} and Lemma \ref{lemma: chains with a constant factor}, respectively, as well as the consequence of the inductive Claim \ref{claim: gammakt is a map of chains} that $\bl_i(C_{2j}\times \bP^{t-j})$ coincides with the space associated to the chain $\{\gamma_{i,j}\times \mathrm{id}\}_{i=0}^{j-1}$. Because $i<j$, the inductive Claim \ref{claim: gammakt is a map of chains} gives
\[ \bl_i(\gamma_{j,t})^{-1}(\bl_i(C_{2i}\times \bP^{t-i}))=\bl_i(C_{2i}\times \bP^{j-i}\times \bP^{t-j}).\]
Then using  \eqref{eqn: functoriality of blow ups} we have
\[ \bl_j(C_{2j}\times \bP^{t-j})\cap E_{i,j}=\bl_{j}(\gamma_{j,t})^{-1}(E_{i,j})=G'_{i,j}. \]
Since $\{G'_{i,j}\}_{i=0}^{j-1}$ is the set of exceptional divisors associated to the NCD chain $\{\gamma_{i,j}\times \mathrm{id}\}_{i=0}^{j-1}$, using Remark \ref{remark: complement of union of divisors} one has
\begin{align*}
\bl_j(C_{2j}\times \bP^{t-j})-\bigcup_{i<j}E_{i,j}&=\bl_j(C_{2j}\times \bP^{t-j})-\bigcup_{i<j}\left(\bl_j(C_{2j}\times \bP^{t-j})\cap E_{i,j}\right)\\
&= \bl_j(C_{2j}\times \bP^{t-j})-\bigcup_{i<j}G_{i,j}'\\
&=C_{2j}\times \bP^{t-j}-(\gamma_{j-1,j}\times \mathrm{id})(C_{2j-2}\times \bP^1\times \bP^{t-j})\\
&=U_{2j}\times \bP^{t-j}.
\end{align*}

\textbf{Step 2}: Claim \ref{claim: blowup restricting to exceptional divisors for
abel jacobi}. Let $i$ be an integer such that $i<j$. We prove Claim \ref{claim: blowup restricting to exceptional divisors for abel jacobi} for all quadruples $(i,\ell,k,t)$ such that $0\leq i<\ell\leq k<t$ and $\ell\in [i+1,j]$, by induction on $\ell$. 

Before the proof, we need some preparation. By the inductive Claim \ref{claim: gammajt is NCD}, the following diagram associated to the chain
$\{\gamma_{i,k}\}_{i=0}^{j-1}$ exists:
\[\begin{tikzcd}
E_{i,j}^k \arrow[r,hook] \arrow[d] &\bl_j(C_{2k}) \arrow[d]\\
\bl_{i}(C_{2i}\times \bP^{k-i}) \arrow[r,hook,"\bl_i(\gamma_{i,k})"] & \bl_i(C_{2k}),
\end{tikzcd}\]
where $E_{i,j}^k$ is the exceptional divisor over $\bl_{i}(C_{2i}\times \bP^{k-i})$. Since $i\leq j-1 $, we can use the inductive Claim \ref{claim: gammakt is a map of chains} to conclude that $\bl_i(C_{2k}\times \bP^{t-k})$ coincides with the space associated to the chain $\{\gamma_{j,k}\times \mathrm{id}\}_{j=0}^{k-1}$ for $i\leq j$.  As $G_{i,j}\subseteq \bl_i(C_{2k}\times \bP^{t-k})$ are exceptional divisors associated to the chain $\{\gamma_{j,k}\times \mathrm{id}\}_{j=0}^{k-1}$, using Lemma \ref{lemma: chains with a constant factor} we have
\begin{equation}\label{eqn: identify Gij with Eij'}
G_{i,j} =  E_{i,j}^k\times \bP^{t-k}.
\end{equation}

Now we start the inductive proof of Claim \ref{claim: blowup restricting to exceptional divisors for abel jacobi}. Let $D\in U_{2i}$, and denote by 
\[M\colonequals \cO_{\bP^1}(g-1-h_{\ast}D).
\]
The base case is $\ell=i+1$.  Let $h<i$. Because the chain $\{\gamma_{h,t}\}_{h=0}^{k-1}$ is NCD and $i\leq k-1$, we know $\bl_i(C_{2i}\times \bP^{t-i})$ intersect with $E_{h,i}$ transversally. Hence $E^{\circ}_{i,i+1}$ is the exceptional divisor over
\begin{equation*}
    \bl_i(C_{2i}\times \bP^{t-i})-\bigcup_{h<i} E_{h,i}\stackrel{\eqref{complement of bljC2j times Pt-j}}{=}U_{2i}\times \bP^{t-i}\subseteq \bl_i(C_{2t}),
\end{equation*} 
thus we may identify $E^{\circ}_{i,i+1}$ as the exceptional divisor for blowing up $C_{2t}$ along $U_{2i}\times \bP^{t-i}$. By Lemma \ref{lemma: conormal of addition maps}(c), we know that $E^{\circ}_{i,i+1}$ is a $U_{2i}$-variety with fiber over $D$ being $B^{t-i-1}(M)$. On the other hand, it follows from the inductive Claim \ref{claim: blowup restricting to exceptional divisors for abel jacobi} and \eqref{eqn: identify Gij with Eij'} that $G_{i,i+1}^{\circ}$ is a $U_{2i}$-variety with fiber over $D$ being $B^{k-i-1}(M)\times \bP^{t-k}$. Moreover, Corollary \ref{corollary: conormal of addition maps} shows that $\bl_{i+1}(\gamma_{k,t}):G^{\circ}_{i,i+1}\to E^{\circ}_{i,i+1}$ is a $U_{2i}$-morphism and the fiber over $D$ is 
\[\alpha_{k-i-1,t-i-1}:B^{k-i-1}(M)\times \bP^{t-k} \to B^{t-i-1}(M). \]
This concludes the base case.

Assume Claim \ref{claim: blowup restricting to exceptional divisors for abel jacobi} holds for all $\ell'\leq \ell-1$. Consider the following diagram
\[ \begin{tikzcd}
G^{\circ}_{i,\ell} \arrow[r,hook] \arrow[d] & \bl_{\ell}(C_{2k}\times \bP^{t-k})
\arrow[d]\arrow[r,"\bl_{\ell}(\gamma_{k,t})"] & \bl_{\ell}(C_{2t}) \arrow[d]
\rar[hookleftarrow] & E^{\circ}_{i,\ell} \arrow[d] \\
G^{\circ}_{i,\ell-1} \arrow[r,hook] & \bl_{\ell-1}(C_{2k}\times \bP^{t-k})
\arrow[r,"\bl_{\ell-1}(\gamma_{k,t})"] & \bl_{\ell-1}(C_{2t}) \rar[hookleftarrow] & E^{\circ}_{i,\ell-1} \\
{} &  \bl_{\ell-1}(C_{2\ell-2}\times \bP^{k-(\ell-1)}\times \bP^{t-k}) \arrow[u,hook,"\bl_{\ell-1}(\gamma_{\ell,k}\times \mathrm{id})"] & \bl_{\ell-1}(C_{2\ell-2}\times \bP^{t-(\ell-1)}) \arrow[u,hook,"\bl_{\ell-1}(\gamma_{\ell-1,t})"] & {}
\end{tikzcd}
\]
Using the inductive Claim \ref{claim: gammajt is NCD}, we know the chain $\{\gamma_{i,t}\}_{i=0}^{j-1}$ is NCD and hence 
$E^{\circ}_{i,\ell}$ is the blow up of $E^{\circ}_{i,\ell-1}$ along $E^{\circ}_{i,\ell-1} \cap \bl_{\ell-1}(C_{2\ell-2}\times \bP^{t-(\ell-1)})$. By inductive Claim \ref{claim: blowup restricting to exceptional divisors for abel jacobi} and Lemma \ref{lemma: fiberwise embedding,pullback etc}(3), we deduce that $E^{\circ}_{i,\ell}$ is a $U_{2i}$-variety and the fiber over $D$ is $\bl_{\ell-i-1}B^{t-i-1}(M)$. Similarly, using \eqref{eqn: identify Gij with Eij'}, the inductive Claim \ref{claim: blowup restricting to exceptional divisors for abel jacobi} and Lemma \ref{lemma: chains with a constant factor}, one can deduce that $G^{\circ}_{i,\ell}$ as a $U_{2i}$-variety with fiber $\bl_{\ell-i-1}(B^{k-i-1}(M)\times \bP^{t-k})$. Moreover, by Corollary \ref{cor: injective maps of chains for secant bundles over rational curves}(a), the map $\bl_\ell(\gamma_{k,t}):G_{i,\ell}^{\circ}\to E_{i,\ell}^{\circ}$ is a $U_{2i}$-morphism with fiber over $D$ being 
\[ \bl_{\ell-i-1}(\alpha_{k-i-1,t-i-1}):\bl_{\ell-i-1}B^{k-i-1}(M)\times \bP^{t-k}\to \bl_{\ell-i-1}B^{t-i-1}(M).\]
This finishes the inductive proof of Claim \ref{claim: blowup restricting to exceptional divisors for abel jacobi}.

\textbf{Step 3}: Claim \ref{claim: Gij and Eij}. For $(\ast)$, using \eqref{eqn: identify Gij with Eij'} and the inductive \eqref{complement of bljC2j times Pt-j}, we can show that 
\begin{align*}
    &\bl_j(C_{2j}\times \bP^{k-j}\times \bP^{t-k})-\bigcup_{i<j}G_{i,j}=(U_{2j}\times \bP^{k-j})\times \bP^{t-k}.
\end{align*}
On the other hand, from the previous step we know that $U_{2j}\times \bP^{t-j}$ is not touched by any blowup asssociated to $\bl_j(\gamma_{k,t})$. Using Lemma \ref{lemma: intersection}(b) we obtain
\begin{align*}
    \bl_j(\gamma_{k,t})^{-1}(U_{2j}\times \bP^{t-j})=\gamma_{k,t}^{-1}(U_{2j}\times \bP^{t-j})=(U_{2j}\times \bP^{k-j})\times \bP^{t-k}.
\end{align*}
Together with \eqref{complement of bljC2j times Pt-j}, this proves $(\ast)$.

As for $(\ast\ast)$, the fiberwise equality over $D\in U_{2i}$ follows from Claim \ref{claim: blowup restricting to exceptional divisors for abel jacobi} and Corollary \ref{cor: injective maps of chains for secant bundles over rational curves}(a). Then we can apply Lemma \ref{lemma: fiberwise embedding,pullback etc}(2) to obtain $(\ast\ast)$.  

Finally, Claim \ref{claim: blowup restricting to exceptional divisors for abel jacobi} together with Corollary  \ref{cor: injective maps of chains for secant bundles over rational curves}(a) and Lemma \ref{lemma: fiberwise embedding,pullback etc}(1) give $(\ast\ast\ast)$.

\textbf{Step 4}: Claim \ref{claim: gammakt is a map of chains}. We apply Proposition \ref{prop: injective map imply smooth chain and normal crossing} to the chain $\{\gamma_{j,t}\}_{j=0}^k$ with
\[ X=C_{2t}, \quad X_k=C_{2k}\times \bP^{t-k}, \quad X_{j,k}=C_{2j}\times \bP^{k-j}\times \bP^{t-k}, \quad f_{j,k}=\gamma_{j,k}\times \mathrm{id}, \quad \phi_k=\gamma_{k,t}.\]
The assumptions in Proposition \ref{prop: injective map imply smooth chain and normal crossing} can be checked as follows:
\begin{enumerate}[(I)]
\item The chain $\{\gamma_{j,k}\times \mathrm{id}\}_{j=0}^{k-1}$ is a NCD chain, by the inductive Claim \ref{claim: gammajt is NCD} and Lemma \ref{lemma: chains with a constant factor}.
\item The map
\begin{align*}
\gamma_{k,t}:\left(C_{2k}\times \bP^{t-k}-(\gamma_{k-1,k}\times \mathrm{id})(C_{2k-2}\times \bP^1\times \bP^{t-k})\right)=U_{2k}\times \bP^{t-k} \to C_{2t}\end{align*}
is an embedding, by Lemma \ref{lemma: intersection} and Remark \ref{remark: complement of union of divisors}.
\item The chain $\{\gamma_{j,t}\}_{j=0}^{k-1}$ is NCD by the inductive Claim \ref{claim: gammajt is NCD}.
\item The inductive Claim \ref{claim: gammakt is a map of chains} gives
\[ \bl_i(\gamma_{k,t})^{-1}(\bl_i(C_{2i}\times \bP^{t-i}))=\bl_i(C_{2i}\times \bP^{k-i}\times \bP^{t-k}), \quad \forall i\leq j-1.\]
Moreover, the conditions $(\ast),(\ast\ast)$ are satisfied by Claim \ref{claim: Gij and Eij}.
\end{enumerate}
Consequently, Proposition \ref{prop: injective map imply smooth chain and normal crossing} gives
\[ \bl_j(\gamma_{k,t})^{-1}(\bl_j(C_{2j}\times \bP^{t-j}))=\bl_j(C_{2j}\times \bP^{k-j}\times \bP^{t-k}).\]
This proves Claim \ref{claim: gammakt is a map of chains}. 

\textbf{Step 5}. Claim \ref{claim: gammajt is NCD}. Running the arguments above for all $j\leq k$ and using $(\ast\ast\ast)$ from Claim \ref{claim: Gij and Eij}, Proposition \ref{prop: injective map imply smooth chain and normal crossing} also implies that $\{\gamma_{j,t}\}_{j=0}^k$ is a NCD chain. This proves Claim \ref{claim: gammajt is NCD}.

Therefore, we finish the inductive proof for Claim \ref{claim: gammajt is NCD} to Claim \ref{claim: Gij and Eij}. As a consequence, this proves Proposition \ref{prop: chain of addition maps}.

\section{The proof of Proposition \ref*{prop: chain of Abel Jacobi maps} and Corollary \ref{corollary: generic structure of exceptional divisors}}\label{sec: proof of chain of abel jacobi maps}

In this section, we prove
Proposition \ref{prop: chain of Abel Jacobi maps} following a similar argument to the one of Proposition \ref{prop: chain of addition maps}. As a byproduct of the proof, we deduce Corollary \ref{corollary: generic structure of exceptional divisors}.

Let $C$ be a smooth hyperelliptic curve of odd genus $g = 2n+1$. Consider the proper chain  $\{\delta_k:C_{2k} \to \Jac(C)\}_{k=0}^n$ from \eqref{eqn: the chain of maps to Jacobian} with the diagram from \eqref{eqn: diagram for abel jacobi maps} for $j<k$
\[ \begin{tikzcd} 
C_{2j}\times \bP^{k-j} \arrow[r,twoheadrightarrow,"p_1"] \arrow[d,"\gamma_{j,k}"]& C_{2j} \arrow[d,"\delta_j"] \\
C_{2k}\arrow[r,"\delta_k"] & \Jac(C)
\end{tikzcd}\]
Let $k$ be an integer such that $0\leq k\leq n$. Denote by $\bl_j(C_{2k})$ and $\bl_j(\Jac(C))$
the spaces associated to the chain $\{\delta_k\}_{k=0}^n$. Denote by
$\bl_i(C_{2j}\times \bP^{k-j})$, the space associated to the chain
$\{\gamma_{j,k}\}_{j=0}^{k-1}$. We break the proof of Proposition \ref{prop: chain of
Abel Jacobi maps} down into the following statements:
\begin{claim}\label{claim: inductive for NCD chain AbelJacobi}
$\{\delta_{j}\}_{j=0}^{k}$ is a NCD chain.
\end{claim}

\begin{claim}\label{claim: abel Jacobi is a map of chains}
Let $j\leq k$, then $\bl_j(\delta_k)^{-1}(\bl_j(C_{2j}))=\bl_j(C_{2j}\times \bP^{k-j})$. In other words, $\delta_k$ is a map of chains of centers.
\end{claim}
Consequently, $\bl_{j}(C_{2k})$ coincides with the blowup space associated to the chain $\{\gamma_{j,k}\}_{j=0}^{k-1}$ and there is a natural isomorphism 
$\bl_j(C_{2j}\times \bP^{k-j})=\bl_j(C_{2j})\times \bP^{k-j}$.

\begin{claim}\label{claim: pullback of Fij and complements}
Let $0\leq i<j\leq k$ and consider the following diagram:
\begin{equation*}
\begin{tikzcd}
{} & {} & \bl_j(C_{2j})\arrow[d,hookrightarrow,"\bl_j(\delta_j)"] & \\
E_{i,j}\arrow[r,hookrightarrow] \arrow[d] &\bl_j(C_{2k}) \arrow[r,"\bl_j(\delta_k)"] \arrow[d] &\bl_j(\Jac(C)) \arrow[d]\arrow[r,hookleftarrow] &F_{i,j}\arrow[d]\\
\bl_i(C_{2i}\times \bP^{k-i}) \arrow[r,hookrightarrow,"\bl_i(\gamma_{i,k})"] &\bl_i(C_{2k}) \arrow[r,"\bl_i(\delta_k)"] &\bl_i(\Jac(C)) \arrow[r,hookleftarrow,"\bl_i(\delta_i)"] &\bl_i(C_{2i})
\end{tikzcd}
\end{equation*}
where $F_{i,j},E_{i,j}$ are the exceptional divisors over $\bl_i(C_{2i})$ and
$\bl_i(C_{2i}\times \bP^{k-i})$ respectively. Then
\begin{equation}\label{eqn: pullback of Fij is Eij}
\bl_{j}(\delta_k)^{-1}(F_{i,j})=E_{i,j},
\end{equation}
\begin{equation}\label{eqn: complement of all exceptional divisors for Jac}
\bl_j(C_{2j}) - \bigcup_{i<j} F_{i,j} = U_{2j},
\end{equation}
where $U_{2j}=C_{2j}-\gamma_{j,j-1}(C_{2j-2}\times \bP^1)$.
\end{claim}

\begin{claim}\label{claim: blowup restricting to exceptional divisors for Jac}
For $0\leq i<j\leq k$, denote by $E_{i,j}^{\circ}=E_{i,j}-\bigcup_{h<i}E_{h,j}$ and the same for $F_{i,j}^{\circ}$. Then \eqref{eqn: pullback of Fij is Eij} induces a morphism of $U_{2i}$-varieties 
\[ \bl_j(\delta_k):E_{i,j}^{\circ}\to F_{i,j}^{\circ}. \]
Moreover, over $D\in U_{2i}$ the morphism is 
\[ \bl_{j-i-1}(\beta_{k-i-1}):\bl_{j-i-1}B^{k-i-1}(M)\to \bl_{j-i-1}\bP H^0(M),\]
which is associated to the chain $\{\beta_\ell\}_{\ell=0}^{k-i-1}$ in \eqref{eqn: the chain of secant varieties} with respect to the line bundle $M=\cO_{\bP^1}(g-1-h_{\ast} D)$ (see Notation \ref{notation: associated divisors on P1}).

\end{claim}

\begin{claim}\label{claim: conditions ast for Abel Jacobi}
Let $0\leq i<j\leq k$, then
\begin{align*}
\bl_j(\delta_k)^{-1}(\bl_j(C_{2j})-\bigcup_{i<j} F_{i,j})=\bl_j(C_{2j}\times \bP^{k-j})-\bigcup_{i<j} E_{i,j},\tag{$\ast$}\\
\bl_j(\delta_k)^{-1}( \bl_j(C_{2j})\cap F_{i,j}^{\circ})=\bl_j(C_{2j}\times \bP^{k-j})\cap E_{i,j}^{\circ},\tag{$\ast\ast$}\\
\bl_k(\delta_k):E_{i,k}^{\circ}\to F_{i,k}^{\circ} \textrm{ is an embedding} \tag{$\ast\ast\ast$}.
\end{align*}
\end{claim}

As in the proof of Proposition \ref{prop: chain of addition maps}, we prove these claims by induction on the tuples $(j,k)$. The base case contains the following cases.
\begin{itemize}
    \item $k=0$: all the claims are clear.
    \item $j=0$ and arbitrary $k\geq 0$ (except Claim \ref{claim: inductive for NCD chain AbelJacobi}). For Claim \ref{claim: abel Jacobi is a map of chains} this is clear, as Abel-Jacobi theorem gives
\[ \bl_0(\delta_k^{-1})(C_0)=\delta_k^{-1}(C_0)=C_0\times \bP^k.\]
The statement for other claims are empty. 
\end{itemize}
As in the previous section, we summarize the assumptions that we are allowed to use
during the inductive step in the following way:
\begin{assumption}
Let $j,k$ be integers such that $j\leq k$. Suppose
\begin{itemize} 
\item Claim \ref{claim: inductive for NCD chain AbelJacobi} holds for all $k'$ such that $k'\leq k-1$.
\item Claim \ref{claim: abel Jacobi is a map of chains} to Claim \ref{claim: conditions ast for Abel Jacobi} hold for all the tuples $(j',k')$ such that $j'\leq k'\leq k-1$.
\item Claim \ref{claim: abel Jacobi is a map of chains} holds for all $(j',k)$ with  $j'\leq j-1$.
\end{itemize}
\end{assumption}
\noindent We will prove that Claim \ref{claim: inductive for NCD chain AbelJacobi}
holds for $k$ and Claim \ref{claim: abel Jacobi is a map of chains} to Claim \ref{claim: conditions ast for Abel Jacobi} hold for the tuple $(j,k)$.

\textbf{Step 1}: Claim \ref{claim: pullback of Fij and complements}. By the inductive
hypothesis, all maps in the diagram are well-defined. Let $i$ be an integer such that $i<j$. The inductive Claim \ref{claim: abel Jacobi is a map of chains} gives
\begin{equation*}\bl_i(\delta_k)^{-1}(\bl_i(C_{2i}))=\bl_i(C_{2i}\times \bP^{k-i}).\end{equation*}
Using the inductive Claim \ref{claim: inductive for NCD chain AbelJacobi} and Claim
\ref{claim: gammajt is NCD} that $\{\delta_j\}$ and $\{\gamma_{i,k}\}$ are chains of
smooth centers, we know $\bl_i(C_{2i})$, $\bl_i(C_{2i}\times \bP^{k-i})$ are smooth
varieties.  Then repeated application of \eqref{eqn: functoriality of blow ups} yields \eqref{eqn: pullback of Fij is Eij}.

Regarding \eqref{eqn: complement of all exceptional divisors for Jac}, to understand the intersection $\bl_j(C_{2j})\cap F_{i,j}$, let us consider the following diagram:
\begin{equation*} \begin{tikzcd}
E'_{i,j}\arrow[r,hookrightarrow] \arrow[d] &\bl_j(C_{2j}) \arrow[r,hookrightarrow,"\bl_j(\delta_j)"] \arrow[d] &\bl_j(\Jac(C)) \arrow[d]\arrow[r,hookleftarrow] &F_{i,j}\arrow[d]\\
\bl_i(C_{2i}\times \bP^{j-i}) \arrow[r,hookrightarrow,"\bl_i(\gamma_{i,j})"] &\bl_i(C_{2j}) \arrow[r,"\bl_i(\delta_j)"] &\bl_i(\Jac(C)) \arrow[r,hookleftarrow,"\bl_i(\delta_i)"] &\bl_i(C_{2i}), 
\end{tikzcd}
\end{equation*}
where $E'_{i,j}$ is the exceptional divisor over $\bl_i(C_{2i}\times \bP^{j-i})$. This diagram exists because $\{\delta_{i}\}_{i=0}^{j-1}$ and
$\{\gamma_{i,j}\}_{i=0}^{j-1}$ are NCD chains by the inductive Claim \ref{claim:
inductive for NCD chain AbelJacobi} and Proposition \ref{prop: chain of addition
maps}, respectively. As $i<j$, the inductive Claim \ref{claim: abel Jacobi is a map of chains} gives
\[ \bl_i(\delta_j)^{-1}(\bl_i(C_{2i}))=\bl_i(C_{2i}\times \bP^{j-i}).\]
As a consequence, using \eqref{eqn: functoriality of blow ups} again we have
\[ \bl_j(C_{2j})\cap F_{i,j}=\bl_j(\delta_j)^{-1}(F_{i,j})= E'_{i,j}.\] By construction, $\{E'_{i,j}\}_{i=0}^{j-1}$ is the set of exceptional divisors associated to the NCD chain $\{\gamma_{i,j}\}_{i=0}^{j-1}$, thus it follows from Remark \ref{remark: complement of union of divisors} that
\begin{equation*}
    \bl_j(C_{2j})-\bigcup_{i<j}F_{i,j}=\bl_j(C_{2j})-\bigcup_{i<j}E'_{i,j}=C_{2j}-\gamma_{j-1,j}(C_{2j-2}\times \bP^1)=U_{2j}.
\end{equation*}

\textbf{Step 2}: Claim \ref{claim: blowup restricting to exceptional
divisors for Jac}. Let $i$ be an integer with $i<j$. We prove Claim
\ref{claim: blowup restricting to exceptional divisors for Jac} for all triples $(i,\ell,k)$ such that $0<i<\ell\leq k$ and  $\ell\in [i+1,j]$, by an additional induction on $\ell$.

The base case is $\ell=i+1$. Let $h<i$. By construction, $F_{i,i+1}$
and $F_{h,i+1}$ are the exceptional divisors over $\bl_i(C_{2i})$ and
$F_{h,i}$, respectively. To understand 
\[F_{i,i+1}^{\circ}=F_{i,i+1}-\cup_{h<i}F_{h,i+1},\] we
need to know how $F_{h,i}$ intersects $\bl_i(C_{2i})$. By the inductive Claim \ref{claim: inductive for NCD chain AbelJacobi}, the chain $\{\delta_h\}_{h=0}^{k-1}$ is NCD, and hence we have transverse intersections
\[ \bl_i(C_{2i})\cap F_{h,i} \subseteq \bl_i(\Jac(C)).\]
It follows that $F^{\circ}_{i,i+1}$ is the exceptional divisor for the blow up of $\bl_i(\Jac(C))$ along
\begin{equation*}
    \bl_i(C_{2i})-\bigcup_{h<i} F_{h,i}=U_{2i},
\end{equation*} 
which  follows from the inductive \eqref{eqn: complement of all exceptional divisors for Jac}. Since $U_{2i}$ is away from all the exceptional divisors $F_{h,i}$, $F^{\circ}_{i,i+1}$ can also be identified with the exceptional divisor for the blow up of $\Jac(C)$ along $U_{2i}$, so that we can use Bertram's results as follows.

Let $D\in U_{2i}$ and write $M=\cO_{\bP^1}(g-1- h_{\ast} D)$. It follows from Lemma \ref{lemma: conormal of abel jacobi maps}(b) that $F^{\circ}_{i,i+1}$ is a $U_{2i}$-variety with fiber over $D$ being $\bP H^0(M)$. On the other hand, Claim \ref{claim: blowup restricting to exceptional divisors for abel jacobi} says that $E^{\circ}_{i,i+1}$ is a $U_{2i}$-variety with fiber over $D$ being $B^{k-i-1}(M)$. Moreover, by Corollary \ref{corollary: conormal of abel jacobi maps}, $\bl_{i+1}(\delta_k):E_{i,i+1}^{\circ} \to F_{i,i+1}^{\circ}$ is a $U_{2i}$-morphism, whose fiber over $D$ is
\[\beta_{k-i-1}:B^{k-i-1}(M) \to \bP H^0(M). \]
We conclude the base case $\ell=i+1$.

Assume Claim \ref{claim: blowup restricting to exceptional divisors for Jac} holds for all $\ell'\leq \ell-1$. Let us look at the diagram
\[ \begin{tikzcd}
E^{\circ}_{i,\ell} \arrow[r,hook] \arrow[d] & \bl_{\ell}(C_{2k})
\arrow[d]\arrow[r,"\bl_{\ell}(\delta_{k})"] & \bl_{\ell}(\Jac(C)) \arrow[d]
\rar[hookleftarrow] & F^{\circ}_{i,\ell} \arrow[d] \\
E^{\circ}_{i,\ell-1} \arrow[r,hook] & \bl_{\ell-1}(C_{2k})
\arrow[r,"\bl_{\ell-1}(\delta_{k})"] & \bl_{\ell-1}(\Jac(C)) \rar[hookleftarrow] & F^{\circ}_{i,\ell-1} \\
{} & \bl_{\ell-1}(C_{2\ell-2}\times \bP^{k-(\ell-1)}) \arrow[u,hook] & \bl_{\ell-1}(C_{2\ell-2}) \arrow[u,hook] & {}
\end{tikzcd}
\]
Moreover, \eqref{eqn: pullback of Fij is Eij} gives the diagram
\begin{equation}\label{diagram from Eiell to Fiell}\begin{tikzcd}
    E^{\circ}_{i,\ell} \arrow[d]\arrow[rr,"\bl_{\ell}(\delta_k)"]& &F^{\circ}_{i,\ell}\arrow[d]\\
    E^{\circ}_{i,\ell-1}\arrow[rr,"\bl_{\ell-1}(\delta_k)"] &&F^{\circ}_{i,\ell-1}
\end{tikzcd}\end{equation}
Let $D\in U_{2i}$, we would like to understand the fiber of the top horizontal map over $D$. To do this, let us compute the other parts of the diagram over $D$.

\textbf{Part I}: $F^{\circ}_{i,\ell-1}$. As $\ell-1\leq k-1$,  the inductive Claim \ref{claim: inductive for NCD chain AbelJacobi} implies that there is an embedding $\bl_{\ell-1}(\delta_{\ell-1}):\bl_{\ell-1}(C_{2\ell-2})\hookrightarrow \bl_{\ell-1}(\Jac(C))$, 
and the intersection 
\[\bl_{\ell-1}(C_{2\ell-2})\cap F^{\circ}_{i,\ell-1}\]
is transverse. Consequently, $F^{\circ}_{i,\ell}$ is
the blowup of $F^{\circ}_{i,\ell-1}$ along $ \bl_{\ell-1}(C_{2\ell-2})\cap F^{\circ}_{i,\ell-1}$. Note that $M=\cO_{\bP^1}(g-1-h_{\ast}D)$ has degree 
\[g-1-2i=2n-2i\geq 2(k-i-1)+1,\]
as $g(C)=2n+1$ and $k\leq n$. By the inductive hypothesis,
\[ F^{\circ}_{i,\ell-1}, \quad \bl_{\ell-1}(C_{2\ell-2})\cap F^{\circ}_{i,\ell-1}=\bl_{\ell-1}(\delta_{\ell-1})^{-1}(F^{\circ}_{i,\ell-1})\]
are $U_{2i}$-varieties whose fibers over $D$ are
\[ \bl_{\ell-i-2}\bP H^0(M),\quad \bl_{\ell-i-2}B^{\ell-i-2}(M).\]
We can use Lemma \ref{lemma: fiberwise embedding,pullback etc}(3) to obtain that $F^{\circ}_{i,\ell}$ is a  $U_{2i}$-variety with fiber over $D$ being the blow up of $\bl_{\ell-i-2}\bP H^0(M)$ along
$\bl_{\ell-i-2}B^{\ell-i-2}(M)$, which is by definition $\bl_{\ell-i-1}\bP H^0(M)$. 

\textbf{Part II}: $E^{\circ}_{i,\ell-1}$. A similar argument using Claim \ref{claim: blowup restricting to exceptional divisors for abel jacobi} says that $E^{\circ}_{i,\ell}$ is a $U_{2i}$-variety with fiber over $D$ being $\bl_{\ell-i-1}B^{k-i-1}(M)$, the blow up of $\bl_{\ell-i-2}B^{k-i-1}(M)$ along $\bl_{\ell-i-2}(B^{\ell-i-2}(M)\times \bP^{k-\ell+1})$. 

\textbf{Part III}: maps in \eqref{diagram from Eiell to Fiell}. The inductive hypothesis says $\bl_{\ell-1}(\delta_k): E^{\circ}_{i,\ell-1} \to F^{\circ}_{i,\ell-1}$
over $D$ is $\bl_{\ell-i-2}(\beta_{k-i-1})$. The vertical maps over $D$ are the blow ups described above.

Consequently, we can draw the corresponding diagram of \eqref{diagram from Eiell to Fiell} over $D$
\[ \begin{tikzcd}
      \bl_{\ell-i-1}B^{k-i-1}(M) \arrow[rr] \arrow[d]&& \bl_{\ell-i-1}\bP H^0(M)\arrow[d]\\
     \bl_{\ell-i-2}B^{k-i-1}(M) \arrow[rr,"\bl_{\ell-i-2}(\beta_{k-i-1})"] &&\bl_{\ell-i-2}\bP H^0(M)
\end{tikzcd}
\]
Since $\beta_{k-i-1}$ is a map of chains (Proposition \ref{prop: injective maps of chains for secant bundles}), the top horizontal map must be $\bl_{\ell-i-1}(\beta_{k-i-1})$, which is the fiber of $\bl_{\ell}(\delta_k):E^{\circ}_{i,\ell}\to F^{\circ}_{i,\ell}$ over $D$. This finishes the inductive proof of Claim \ref{claim: blowup restricting to exceptional divisors for Jac}.

\textbf{Step 3}: Claim \ref{claim: conditions ast for Abel Jacobi}. For
$(\ast)$, we know from the discussion in Step 2 that $U_{2j}$ is not touched by all the blowups associated to $\bl_j(\delta_k)$, thus
\begin{equation*}
\bl_j(\delta_k)^{-1}(U_{2j})=\delta_{k}^{-1}(U_{2j})=U_{2j}\times \bP^{k-j},
\end{equation*}
where the last equality comes from Lemma \ref{lemma: intersection}(a). 
Using \eqref{complement of bljC2j times Pt-j} in Claim \ref{claim: pullback of Eij and complements} we have
\begin{align*}
    \bl_j(C_{2j}\times \bP^{k-j})-\bigcup_{i<j}E_{i,j}=U_{2j}\times \bP^{k-j}.
\end{align*}
Putting these two equations together with \eqref{eqn: complement of all exceptional divisors for Jac}, we obtain $(\ast)$.

Let us turn to $(\ast\ast)$. Let $D\in U_{2i}$. By Claim \ref{claim: blowup restricting to exceptional divisors for Jac} and Claim \ref{claim: blowup restricting to exceptional divisors for abel jacobi}, the diagram
\begin{equation}\label{diagram: claim 6.5 astast}
\begin{tikzcd}
\bl_j(C_{2j}\times \bP^{k-j})   \cap E^{\circ}_{i,j} \arrow[r] \arrow[d,hook]&  \bl_j(C_{2j}) \cap F^{\circ}_{i,j} \arrow[d,hook] \\
E^{\circ}_{i,j} \arrow[r,"\bl_j(\delta_k)"] & F^{\circ}_{i,j}
\end{tikzcd}\end{equation}
is a diagram of $U_{2i}$-varieties, whose fiber over  $D$ is
\begin{equation}\label{diagram: fiberwise claim 6.5 ast ast}\begin{tikzcd}[column sep=large] 
\bl_{j-i-1}(B^{j-i-1}(M)\times \bP^{k-j}) \arrow[r] \arrow[d,hook,"\bl_{j-i-1}(\alpha_{j-i-1,k-i-1})"]& \bl_{j-i-1}B^{j-i-1}(M) \arrow[d,hook,"\bl_{j-i-1}(\beta_{j-i-1})"] \\
\bl_{j-i-1}B^{k-i-1}(M) \arrow[r,"\bl_{j-i-1}(\beta_{k-i-1})"] & \bl_{j-i-1}\bP H^0(M)
\end{tikzcd}\end{equation}
where $M$ denotes the line bundle
$\cO_{\bP^1}(g-1-h_{\ast} D)$ and the top horizontal map is induced by the natural projection
\[ B^{j-i-1}(M)\times \bP^{k-j} \to B^{j-i-1}(M).\]
By Corollary \ref{cor: injective maps of chains for secant bundles over rational curves}(b), the diagram \eqref{diagram: fiberwise claim 6.5 ast ast} is Cartesian, i.e.
\[ \bl_{j-i-1}(\beta_{k-i-1})^{-1}(\bl_{j-i-1}B^{j-i-1}(M))=\bl_{j-i-1}(B^{j-i-1}(M)\times \bP^{k-j}).\]
As $(\bl_j(C_{2j})\times \bP^{k-j})\cap E^{\circ}_{i,j}$ and $\bl_j(C_{2j})\cap F^{\circ}_{i,j}$ are both smooth, we can apply Lemma \ref{lemma: fiberwise embedding,pullback etc}(2) to the $U_{2i}$-morphism $\bl_j(\delta_k):E^{\circ}_{i,j}\to F^{\circ}_{i,j}$ to conclude that the diagram \eqref{diagram: claim 6.5 astast} is also Cartesian. This proves $(\ast\ast)$.

Last, for $(\ast\ast\ast)$. Claim \ref{claim: blowup restricting to exceptional divisors for Jac} implies  that the fiber of the map $\bl_{k}(\delta_k):E^{\circ}_{i,k}\to F^{\circ}_{i,k}$ over $D\in U_{2i}$ is $\bl_{k-i-1}(\beta_{k-i-1})$, which is an embedding by Corollary \ref{cor: injective maps of chains for secant bundles over rational curves}(b). As $E_{i,k}^{\circ},F^{\circ}_{i,k}$ are both smooth, we conclude from Lemma \ref{lemma: fiberwise
embedding,pullback etc}(1) that $(\ast\ast\ast)$ holds.

\textbf{Step 4}: Claim \ref{claim: abel Jacobi is a map of chains}. We apply Proposition \ref{prop: injective map imply smooth chain and normal crossing} to the chain $\{\delta_j\}_{j=0}^k$ with
\[ X=\Jac(C), \quad X_k=C_{2k}, \quad X_{j,k}=C_{2j}\times \bP^{k-j}, \quad f_{j,k}=\gamma_{j,k}, \quad \phi_k=\delta_k.\]
The assumptions in Proposition \ref{prop: injective map imply smooth chain and normal crossing} can be checked as follows:
\begin{enumerate}[(I)]
\item The chain $\{\gamma_{j,k}\}_{j=0}^{k-1}$ is NCD, by Proposition \ref{prop: chain of addition maps}.
\item The map $\delta_k:C_{2k}-\gamma_{k-1,k}(C_{2k-2}\times \bP^1)\to \Jac(C)$ is an embedding: by Notation \ref{notation: open part of symmetric product U2j} $C_{2k}-\gamma_{k-1,k}(C_{2k-2}\times \bP^1)=U_{2k}$ and $U_{2k}\hookrightarrow \Jac(C)$ by Lemma \ref{lemma: intersection}(a).
\item The chain $\{\delta_j\}_{j=0}^{k-1}$ is NCD, by the inductive Claim \ref{claim: inductive for NCD chain AbelJacobi}.
\item $\bl_i(\delta_k)^{-1}(\bl_i(C_{2i}))=\bl_i(C_{2i}\times \bP^{k-i})$ for all $i\leq j-1$, by the inductive Claim \ref{claim: abel Jacobi is a map of chains}.
\item The conditions $(\ast),(\ast\ast),$ are satisfied thanks to Claim  \ref{claim: conditions ast for Abel Jacobi}.
\end{enumerate}
Hence Proposition \ref{prop: injective map imply smooth chain and normal crossing} shows that
\[ \bl_j(\delta_k)^{-1}(\bl_j(C_{2j}))=\bl_j(C_{2j}\times \bP^{k-j}),\]
i.e. Claim \ref{claim: abel Jacobi is a map of chains} holds. 

\textbf{Step 5}: Claim \ref{claim: inductive for NCD chain AbelJacobi}. Running the arguments above for all $j\leq k$ and using $(\ast\ast\ast)$ from Claim \ref{claim: conditions ast for Abel Jacobi}, Proposition \ref{prop: injective map imply smooth chain and normal crossing} implies that that $\{\delta_j\}_{j=0}^k$ is a NCD chain. This proves Claim \ref{claim: inductive for NCD chain AbelJacobi}.
 
Therefore, we finish the inductive proof of all the claims above. Proposition \ref{prop: chain of Abel Jacobi maps} immediately follows.

\begin{proof}[Proof of Corollary \ref{corollary: generic structure of exceptional divisors}]
For $0\leq i \leq n-1$, the exceptional divisor $Z_i\subseteq \bl_{n+1}(\Jac(C))$ is the exceptional divisor $F_{i,n+1}\subseteq \bl_{n+1}(\Jac(C))$, defined as the preimage of $F_{i,n}\subseteq \bl_{n}(\Jac(C))$ (see the notations in Claim \ref{claim: pullback of Fij and complements}). The proper transform $\tilde{\Theta}$ is $F_{n,n+1}$, the exceptional divisor for the blow up of $\bl_n(\Jac(C))$ along $\bl_n(C_{2n})$. Since $\delta_{i}(C_{2i})=W^{n-i}_{g-1}$ and $\delta_i$ is
an isomorphism over $U_{2i}\subseteq C_{2i}$ by Lemma \ref{lemma: intersection}, this induces a natural identification
\[ U_{2i}\cong W^{n-i}_{g-1}(C)-W^{n-i+1}_{g-1}(C). \]

Let $D\in U_{2i}$ and set $M=\cO_{\bP^1}(g-1-h_{\ast}D)$. Claim \ref{claim: blowup restricting to exceptional divisors for Jac} implies that the fiber of the projection 
\[ F_{i,n}^{\circ}=F_{i,n}- \bigcup_{h<i} F_{h,n} \to W^{n-i}_{g-1}(C)-W^{n-i+1}_{g-1}(C)\cong U_{2i}\]
over $D$ is $\bl_{n-i-1}\bP H^0(\bP^1,M)$, the blowup space associated to the chain 
\[\{\beta_{k-i-1}:B^{k-i-1}(M)  \to \bP H^0 (M)\}_{k=i+1}^{n}.\]
Blowing up one more time, we can argue as in the proof of Claim \ref{claim: blowup restricting to exceptional divisors for Jac} that the fiber of $F_{i,n+1}^{\circ}=F_{i,n+1}- \bigcup_{h<i} F_{h,n+1}$ over $D$ is $\bl_{n-i}\bP H^0(\bP^1,M)$.
Let $k$ be an integer such $i+1\leq k\leq n$, Claim \ref{claim: blowup restricting to exceptional divisors for Jac} implies that
\[F^{\circ}_{i,k}\cap \bl_{k}(C_{2k})=E^{\circ}_{i,k}\subseteq \bl_k(\Jac(C))\]
is a $U_{2i}$-variety with fiber over $D$ is $\bl_{k-i-1}B^{k-i-1}(M)$. Then in the final blow up space we have
\[ F_{i,n+1}^{\circ} \cap F_{k,n+1}\subseteq \bl_{n+1}(\Jac(C))\]
is a $U_{2i}$-variety with fiber over $D$ is the exceptional divisor 
\[ H_{k-i-1}\subseteq  \bl_{n-i}\bP H^0(\bP^1,M)\]
over $\bl_{k-i-1}B^{k-i-1}(M)$, associated to the chain $\{\beta_{k-i-1}\}_{k=i+1}^{n}$. 

As $\{\beta_{k-i-1}\}_{k=i+1}^n$ is a smooth chain by Corollary \ref{cor: injective maps of chains for secant bundles over rational curves}(b), Remark \ref{remark: complement of union of divisors} implies that the fiber of
\[ Z_{i}-\bigcup_{\substack{\quad 0\leq j\leq n-1\\j\neq i}}Z_j-\tilde{\Theta}=F_{i,n+1}^{\circ}-\bigcup_{i+1\leq k\leq n} F_{k,n+1} \]
over $D$ is
\[ \bl_{n-i}\bP H^0(\bP^1,M)-\bigcup_{i+1\leq k\leq n}H_{k-i-1}\cong \bP H^0(\bP^1,M)- \beta_{n-i-1}(B^{n-i-1}(M)).\]
By definition, the image $\beta_{n-i-1}(B^{n-i-1}(M))$ is the $(n-i-1)$-th secant variety 
\[\mathrm{Sec}^{n-i-1}(\bP^1)\subseteq 
\bP H^0(\bP^1,M)=\bP^{2(n-i)},\]
where $\bP^1$ is the rational normal curve $\beta_0(B^0(M))$. Note that the secant variety $\mathrm{Sec}^{n-i-1}(\bP^1)$ has degree $(n-i)+1$ and the rational curve has degree $2(n-i)$. Applying $i=n-r$, we obtain Corollary \ref{corollary: generic structure of exceptional divisors}.
\end{proof}

The following result is an (almost immediate) consequence of the proof above.
\begin{corollary}
    Let $L\in W^{r}_{g-1}(C)-W^{r+1}_{g-1}(C)$ for some $0\leq r\leq n$, so that $L=\delta_{n-r}(D)$ for a unique $D\in U_{2(n-r)}$. Set $M=\cO_{\bP^1}(g-1-h_{\ast}D)$, then the fiber of the map
    \[  \bl_{n+1}(\Jac(C)) \to \Jac(C)\]
    over $L$ is $\bl_{r}\bP H^0(\bP^1,M)$. In particular, the fiber of the map
	 $\pi_n:\bl_n(\Jac(C)) \to \Jac(C)$ in Theorem \ref{theorem: log resolution of
	 hyperelliptic theta divisors} over $L$ is $\bl_{r}\bP H^0(\bP^1,M)$, using
	 Bertram's notation from \S\ref{sec: Bertram's blow up construction}.

\end{corollary}

\begin{proof}
   Set $i=n-r$. The fiber we are interested in is the fiber of $F_{i,n+1}$ over $D$, using the notations above. From the proof of Corollary \ref{corollary: generic structure of exceptional divisors} we know that the fiber of 
    \[ F^{\circ}_{i,n+1} = F_{i,n+1} - \bigcup_{h<i} F_{h,n+1}\]
    over $D$ is $\bl_{n-i}\bP H^0(\bP^1,M)$. We claim that the fiber of $F_{i,n+1}$ over $D$ is the same as the fiber of $F^{\circ}_{i,n+1}$ over $D$. The reason is as follows. Let $h<i$, \eqref{eqn: complement of all exceptional divisors for Jac} gives $\bl_i(C_{2i})- \cup_{h<i} F_{h,i}=U_{2i}$. Hence the image of $U_{2i}\hookrightarrow \bl_i(C_{2i})\hookrightarrow \bl_i(\Jac(C))$ is away from $F_{h,i}$. As $F_{h,n+1}$ is the exceptional divisor over $F_{h,i}$, we conclude that the fiber of $F_{h,n+1}$ over $D$ is empty. This proves the desired result.

    The last statement follows from the fact that the image of $\bl_n(C_{2n})$ in $\bl_n(\Jac(C))$ is a divisor, so $\bl_{n+1}(\Jac(C))=\bl_n(\Jac(C))$.
\end{proof}
\section{Even genus case}\label{sec: even genus case}

In this section, let $C$ be a smooth hyperelliptic curve of even genus $g=2n+2$. We sketch a proof of Theorem \ref{theorem: log resolution of hyperelliptic theta divisors} for $C$. The ideas are essentially the same by reducing to the calculation of conormal bundles, but for parity reasons, the corresponding maps need some modification. First, we have a chain of maps $\{\delta_j\}_{j=0}^n$ to $\Jac(C)$, where
\begin{align*}
    \delta_{j}:C_{2j+1} &\to \Pic^{g-1}(C)=\Jac(C) ,\quad 0\leq j\leq n \\
    D &\mapsto  (n-j)g^1_2\otimes \cO_C(D).
\end{align*}
The image $\delta_j(C_{2j+1})$ is $W^{n-j}_{g-1}$, hence this is a proper chain. By
the Abel-Jacobi theorem, for each $\ell\geq 1$, we have $\bP^{\ell}\subseteq C_{2\ell}$, then for each $j\geq 1$, there is a proper chain of maps $\{\gamma_{i,j}\}_{i=0}^{j-1}$ induced by the addition maps:
\begin{align*}
    \gamma_{i,j}: C_{2i+1}\times \bP^{j-i} \hookrightarrow C_{2i+1}\times C_{2j-2i}\to C_{2j+1}, \quad 0\leq i<j.
\end{align*}
The even genus case of Theorem \ref{theorem: log resolution of hyperelliptic theta divisors} is reduced to the following analogue of Proposition \ref{prop: chain of addition maps} and Proposition \ref{prop: chain of Abel Jacobi maps}.
\begin{prop}\label{prop: injective maps for even genus} \leavevmode
\begin{enumerate}
    \item For each $1\leq j\leq n$, the chain \[\{\gamma_{i,j}:C_{2i+1}\times \bP^{j-i} \to C_{2j+1}\}_{i=0}^{j-1}\]
is a NCD chain and for each $1\leq i<j$, the map $\gamma_{i,j}$ is a map of chains of centers.
    \item The chain $\{\delta_{j}:C_{2j+1}\to \Jac(C)\}_{j=0}^n$ is a NCD chain and for each $1\leq j\leq n$, the map $\delta_{j}:C_{2j+1} \to \Jac(C)$ is a map of chains of centers.
\end{enumerate}
\end{prop}
\noindent As in the proof of Proposition \ref{prop: chain of addition maps} and
Proposition \ref{prop: chain of Abel Jacobi maps}, the proof of Proposition
\ref{prop: injective maps for even genus} relies on a parallel statement for
Abel-Jacobi maps and conormal bundles as in Lemma \ref{lemma: intersection} and Lemma
\ref{lemma: conormal of addition maps}. We will only give the statements, and leave the proofs to the interested reader.

For each $j\geq 0$, denote
\[ U_{2j+1}\colonequals C_{2j+1}-\gamma_{j-1,j}(C_{2j-1}\times \bP^1).\]
Note that for $j=0$, we have $U_1=C$. Any divisor $D\in U_{2j+1}$ gives a degree $2j+1$ divisor on $\bP^1$ via the hyperelliptic map $C\to \bP^1$, which is denoted to be $h_{\ast}D$ and the associated line bundle is
\[ \cO_{\bP^1}(g-1-h_{\ast}D)\colonequals \cO_{\bP^1}(g-1)\otimes\cO_{\bP^1}(-h_{\ast}D).\]

\begin{lemma}
Let $C$ be a hyperelliptic curve of genus $g=2n+2$. Then:
\begin{enumerate}
      \item [(a)] for any $0\leq i<j$, the map $\gamma_{i,j}:C_{2i+1}\times \bP^{j-i} \to C_{2j+1}$ is an embedding when restricting to $U_{2i+1}\times \bP^{j-i}$ and for $0\leq \ell <i<j$ we have
    \[ \gamma_{i,j}^{-1}(U_{2\ell+1}\times \bP^{j-\ell})=U_{2\ell+1}\times \bP^{i-\ell}\times \bP^{j-i}.\]
    \item [(b)] for $0\leq j\leq n$, the map $\delta_{j}:C_{2j+1} \to \Jac(C)$ is an embedding when restricting to $U_{2j+1}$, and for $0\leq i<j$, we have
    \[ \delta_j^{-1}(U_{2i+1})=U_{2i+1}\times \bP^{j-i}.\]
    In particular, since $U_1=C$ we have
    \[ \delta_j^{-1}(C)=C\times \bP^{j}\subseteq C_{2j+1}.\]

\end{enumerate}

\end{lemma}
\begin{lemma}\label{lemma: conormal of addition maps even genus}
For $0\leq i<j$, consider the map $\gamma_{i,j}:C_{2i+1}\times \bP^{j-i} \to C_{2j+1}$ and let $D\in U_{2i+1}$. Then
\begin{enumerate}
	\item [(a)] $d\gamma_{i,j}: \gamma_{i,j}^{\ast} \ctbl{C_{2j+1}} \to
		\ctbl{C_{2i+1}\times \bP^{j-i}}$ is surjective when restricted to $U_{2i+1}\times \bP^{j-i}$.
    \item [(b)] $\bP N^{\ast}_{\gamma_{i,j}}|_{U_{2i+1}\times \bP^{j-i}}$ is a smooth variety over $U_{2i+1}$ such that over $D\in U_{2i+1}$ we have an isomorphism
    \[\bP N^{\ast}_{\gamma_{i,j}}|_{\{D\}\times \bP^{j-i}}\cong B^{j-i-1}(\cO_{\bP^1}(g-1-h_{\ast}D)). \]

\end{enumerate}    
Furthermore, for $\ell<i$, consider the commutative diagram
\[ \begin{tikzcd}
(C_{2\ell+1}\times \bP^{i-\ell}) \times \bP^{j-i} \arrow[d,"\gamma_{\ell,i}\times \id"] \arrow[r,twoheadrightarrow,"\id\times r"] & C_{2\ell+1}\times \bP^{j-\ell}  \arrow[d,"\gamma_{\ell,j}"]\\
C_{2i+1}\times \bP^{j-i} \arrow[r,"\gamma_{i,j}"] & C_{2j+1}
\end{tikzcd}
\]
Here $r$ is the addition map. For any $D\in U_{2\ell+1}$, there is an induced map of conormal bundles on $\{D\}\times \bP^{i-\ell}\times \bP^{j-i}$:
    \[ \epsilon:(\id\times r)^{\ast}N^{\ast}_{\gamma_{\ell,j}}|_{\{D\}\times \bP^{j-\ell}} \to N^{\ast}_{(\gamma_{\ell,i}\times \id)}|_{\{D\}\times \bP^{i-\ell}\times \bP^{j-i}}.\]
Then:
\begin{enumerate}
    \item [(c)]  $\epsilon$ is surjective.
    \item [(d)] The following diagram commutes:
    \[ \begin{tikzcd}
    \bP N^{\ast}_{(\gamma_{\ell,i}\times \id)}|_{\{D\}\times \bP^{i-\ell}\times \bP^{j-i}} \arrow[r,"\alpha"] \arrow[d,"\cong"] & \bP N^{\ast}_{\gamma_{\ell,i}}|_{\{D\}\times \bP^{j-\ell}} \arrow[d,"\cong"] \\
    B^{i-\ell-1}(M)\times \bP^{j-i} \arrow[r,"\alpha_{i-\ell-1,j-\ell-1}"] & B^{j-\ell-1}(M)
    \end{tikzcd}\]
 Here $\alpha_{i-\ell-1,j-\ell-1}$ is the map \eqref{eqn: alpha_m,k} for the curve  $\bP^1$ and the line bundle  $M=\cO_{\bP^1}(g-1-h_{\ast}D)$; and $\alpha$ is the map induced by $\epsilon$ composed with a projection to $\bP N^{\ast}_{\gamma_{\ell,j}}(D\times \bP^{j-\ell})$.
\end{enumerate}
\end{lemma}

\begin{lemma}\label{lemma: conormal of abel jacobi maps even genus}
With the notation in Lemma \ref{lemma: conormal of addition maps even genus}. For $0\leq j\leq n$, consider the map $\delta_j: C_{2j+1} \to \Jac(C)$. Then:
\begin{enumerate}
	\item [(a)]$d\delta_{j}:\delta_j^{\ast} \ctbl{\Jac(C)} \to \ctbl{C_{2j+1}}$ is surjective when restricted to $U_{2j+1}$.
    \item [(b)] the fiber of $N^{\ast}_{\delta_{j}}$ over $D\in U_{2j+1}$ is $ H^0(\bP^1,\cO_{\bP^1}(g-1-h_{\ast}D))$.
\end{enumerate}    
Furthermore, consider the diagram for $i<j$,
\[ \begin{tikzcd}
C_{2i+1} \times \bP^{j-i} \arrow[d,"\gamma_{i,j}"] \arrow[r,twoheadrightarrow,"p_1"] & C_{2i+1} \arrow[d,"\delta_{i}"]\\
C_{2j+1} \arrow[r,"\delta_{j}"] & 
\Jac(C)
\end{tikzcd}
\]
Then for $D\in U_{2i+1}$, we get the induced map of conormal bundles over $\{D\}\times \bP^{j-i}$:
\[ \epsilon:p_1^{\ast}N^{\ast}_{\delta_{i}}|_D \to N^{\ast}_{\gamma_{i,j}}|_{\{D\}\times \bP^{j-i}},\]
and we have
\begin{enumerate}
   \item [(c)] $\epsilon$ is surjective.
   \item [(d)] The following diagram commutes:
   \[\begin{tikzcd}
   \bP N^{\ast}_{\gamma_{i,j}}|_{\{D\}\times \bP^{j-i}} \arrow[r,"\beta"] \arrow[d,"\cong"]& \bP N^{\ast}_{\delta_{i}}|_{D} \arrow[d,"\cong"] \\
     B^{j-i-1}(M) \arrow[r,"\beta_{j-i-1}"] & \bP H^0(\bP^1,M)
   \end{tikzcd} 
   \]
   The first vertical isomorphism comes from (b) of Lemma \ref{lemma: conormal of addition maps even genus} and $\beta_{j-i-1}$ is the map \eqref{eqn: beta} for the curve $\bP^1$ and the line bundle $M=\cO_{\bP^1}(g-1-h_{\ast}D)$. The map $\beta$ is induced from $\epsilon$ composed with a projection to $\bP N^{\ast}_{\delta_i}|_{D}$.
\end{enumerate}

\end{lemma}

\section{Brill-Noether Stratifications are Whitney}\label{sec: Whitney stratification}
In this section, let $C$ be a smooth projective hyperelliptic curve of genus $g$. We show that the Brill-Noether stratification of $\Jac(C)$ determined by
\[\Jac(C)\supseteq \Theta=W_{g-1}(C) \supseteq W^{1}_{g-1}(C) \supseteq \ldots \supseteq W^{n}_{g-1}(C), \]
is a Whitney stratification
, where $n=\lfloor \frac{g-1}{2} \rfloor$. We will assume $g=2n+1$; the even genus case is similar.

\subsection{Whitney stratifications}

Let $Z$ be a smooth real manifold and let $X,Y\subseteq Z$ be two embedded smooth real sub-manifolds such that $X,Y$ are from a stratification of $Z$. Suppose $Y\subseteq \overline{X}$, where the closure is taken inside $Z$ with respect to the Euclidean topology. 

\begin{definition}
We say that the pair $(X,Y)$ satisfies the \emph{Whitney conditions} if for any point
$y\in Y$ the following two conditions hold: 
\begin{enumerate}
	 \item[(A)] If $\{x_i\}\subseteq X$ is a sequence of points converging to $y$, and
		 if the sequence of tangent spaces $T_{x_i}X$ converges to a linear space $T$ of
		 the same dimension, then $T_{y}Y\subseteq T$.
    \item[(B)] If $\{x_i\}\subseteq X$ and $\{y_i\}\subseteq Y$ are two sequences of
		 points that both converge to $y$, if the sequence of real secant lines
		 between $x_i$ and $y_i$ converges to a real line $L$, and if the sequence of
		 tangent spaces $T_{x_i}X$ converges to a linear subspace $T$ of the same
		 dimension, then $L\subseteq T$.
\end{enumerate}
\end{definition}

The Whitney condition (B) involves real secant lines (in local coordinates), and is
therefore not so easy to verify in practice. Instead, in the case of complex algebraic
varieties, there is a condition (W) introduced by Kuo \cite{Kuo} and Verdier
\cite{Verdier76}, which implies the Whitney conditions and is easier to work with
in our situation. It is proved by Teissier that, for complex analytic
stratifications, the Whitney conditions are equivalent to condition (W), but we will
not need this fact.

\begin{definition}[Distance]
Let $V$ be a complex vector space and let $A,B\subseteq V$ be two linear subspaces.
Fix an inner product $(-,-)$ on $V$. The \emph{distance from $A$ to $B$} is
defined to be
\[ d(A,B)\colonequals \sup_{\substack{a\in A,\\ a\neq 0}}\inf_{\substack{b\in B,\\b\neq 0}} \sin \theta(a,b).\]
Here $\theta(a,b)$ is the angle between two vectors $a,b$ determined by the inner product $(-,-)$.
\end{definition}
Here are some basic properties of $d(A,B)$. Note that it is \emph{not} symmetric in
$A$ and $B$ (which is why we don't call it the distance ``between'' $A$ and $B$).
\begin{fact}\label{fact about distance function}-
\begin{enumerate}
    \item $d(A,B)=0$ if and only if $A\subseteq B$.
    \item Let $A\subseteq A'$ be two subspaces, then $d(A,B)\leq d(A',B)$.
    \item Identify $V$ with the conjugate dual space $V^{\ast}$ via the inner product $(-,-)$ so that $\mathrm{ker} (V^{\ast}\to B^{\ast})$ is identified with the orthogonal complement $B^{\perp}$. Then    \[ d\left(\mathrm{ker} (V^{\ast}\to B^{\ast}),\mathrm{ker} (V^{\ast}\to A^{\ast})\right)=d(B^{\perp},A^{\perp})=d(A,B).\]
		 After choosing an orthonormal basis, this comes down to the fact that a linear
		 operator and its adjoint (between two finite dimensional Hilbert spaces) have
		 the same operator norm.
\end{enumerate}
\end{fact}

From now on, let $Z$ be a complex manifold. Let $X,Y$ be two embedded smooth complex submanifolds of $Z$ such that $Y\subseteq \overline{X}$.
\begin{definition}
We say that the pair $(X,Y)$ satisfies \emph{Condition $(W)$} if for any point $y\in
Y$, and for any sequence of points $\{x_i\}\subseteq X$ converging to $y$, there
exists a constant $C>0$ such that 
\[ d(T_{y}Y,T_{x_i}X)\leq C\cdot d(y,x_i), \quad \forall i,\]
where we view $T_{x_i}X$ as a subspace of $T_yZ$ using a local trivialization of the tangent bundle $T_Z$, and $d(y,x_i)$ is the Euclidean distance between $y$ and $x_i$ in a local coordinate chart.
\end{definition}

Kuo \cite{Kuo} (see also Verdier \cite[Théorème 1.5]{Verdier76}) proved the
following.

\begin{thm}\label{prop: whitney condtions A and B are equivalent to W}
Let $Z$ be a complex manifold. Let $X,Y$ be two embedded smooth complex submanifolds of $Z$ such that $Y\subseteq \overline{X}$ and $X,Y$ are from a stratification of $Z$. If the pair $(X,Y)$ satisfies Condition $(W)$, then the pair $(X,Y)$ satisfies the Whitney conditions $(A),(B)$.
\end{thm}

We are going to use this result in the following form.

\begin{lemma}\label{lemma: the strong condition W'}
Same assumptions as above. Assume the pair $(X,Y)$ satisfies the Whitney condition $(A)$. 
Let $y \in Y$ be an arbitrary point, and let $\{x_i\}\subseteq X$ be a sequence of points
converging to $y$ such that $T=\lim_{i\to \infty}T_{x_i}X$ exists. Suppose that there is a
constant $C >0$ (which is allowed to depend on $y$ and $x_i$) such that
\begin{equation}\label{eqn: the condition W'}
    d(T,T_{x_i}X)\leq C\cdot d(y,x_i), \quad \forall i;
\end{equation} \label{eqn: the condition W', dual statement}
an equivalent formulation is that there is a constant $C > 0$ such that
\begin{equation}
    d((N^{\ast}_{X|Z})_{x_i},\lim_{i\to \infty}(N^{\ast}_{X|Z})_{x_i})\leq C\cdot
	 d(y,x_i), \quad \forall i,
\end{equation} 
where $N^{\ast}_{X|Z}$ denotes the conormal bundle of $X$ inside $Z$.
Then the pair $(X,Y)$ satisfies the Whitney condition $(B)$.
\end{lemma}

\begin{proof}
By Whitney condition $(A)$, we know that $T_{y}Y\subseteq T$. By the property (2) of the distance function  in Fact \ref{fact about distance function}, we conclude that 
\[d(T_{y}Y,T_{x_i}X)\leq d(T,T_{x_i}X)\leq C\cdot d(y,x_i).\]
This verifies the Condition (W) and thus gives the Whitney condition (B) by Theorem \ref{prop: whitney condtions A and B are equivalent to W}. The last statement uses the property (3) of the distance function in Fact \ref{fact about distance function}.
\end{proof}
\begin{definition}\label{definition: a stratification is Whitney}
Let $X$ be a complex algebraic variety and suppose there is a finite algebraic
stratification
\[X = \bigsqcup S_i\]
by connected algebraic varieties whose irreducible components are smooth.
We say this is a Whitney stratification if for any $S_j\subseteq \overline{S_i}$, the pair $(S_i,S_j)$ satisfies the Whitney conditions (A) and (B).
\end{definition}

\subsection{Brill-Noether stratifications are Whitney}
Recall that $C$ is a genus $2n+1$ smooth hyperelliptic curve. For each $0\leq r\leq n$, denote
\[ W^r_{g-1}(C)^{\circ}\colonequals W^r_{g-1}(C)-W^{r+1}_{g-1}(C),\]
which is a connected smooth algebraic variety, and parametrizes degree $g-1$ line bundles with exactly $r+1$ independent sections. The subvariety $\Jac(C)-\Theta$ is also smooth and parametrizes degree $g-1$ lined bundles with no sections. The Brill-Noether stratification of $\Jac(C)$ is defined to be
\[ \Jac(C)=(\Jac(C)-\Theta)\sqcup \bigsqcup_{0\leq r\leq n} W^r_{g-1}(C)^{\circ}. \]
\begin{prop}
The Brill-Noether stratification of $\Jac(C)$ is a Whitney stratification.
\end{prop}
\begin{proof}
Note that $\overline{W^i_{g-1}(C)^{\circ}}=W^i_{g-1}(C)$ and for $i<j$, we have
\[ W^j_{g-1}(C)^{\circ}\subseteq W^j_{g-1}(C)\subseteq W^{i}_{g-1}(C).\]
We also have $\overline{\Jac(C)-\Theta}=\Jac(C)$. By Definition \ref{definition: a
stratification is Whitney}, it suffices to show that for each $i<j$, the pair
$(W^i_{g-1}(C)^{\circ},W^j_{g-1}(C)^{\circ})$ satisfies the Whitney conditions,  and the same holds for the pair $(\Jac(C)-\Theta,W^i_{g-1}(C)^{\circ})$. To apply Lemma \ref{lemma: the strong condition W'}, we need to understand the conormal bundles of the Brill-Noether strata. Recall that for each $0\leq r\leq n$, the Abel-Jacobi map $\delta_{(g-1-2r)/2}=\delta_{n-r}$ induces an isomorphism 
\begin{equation}\label{eqn: the isomorphism between U and open part of BrillNoether}
\delta_{(g-1-2r)/2}:U_{g-1-2r} \xrightarrow{\sim} W^r_{g-1}(C)^{\circ}, \quad D\mapsto \cO_C(D)\otimes rg^1_2,
\end{equation}
where $U_{g-1-2r}$ is defined in Notation \ref{notation: open part of symmetric product U2j} and is the open subset of $C_{g-1-2r}$ consisting of divisors $D$ such that $h^0(C,\cO_C(D))=1$. By Lemma \ref{lemma: conormal of abel jacobi maps}, for any $D\in U_{g-1-2r}$ and $L\colonequals \cO_C(D)\otimes rg^1_2$, we have
\begin{align*} 
(N^{\ast}_{W^r_{g-1}(C)^{\circ}|\Jac(C)})_{L}=(N^{\ast}_{\delta_{(g-1-2r)/2}})_{D}=H^0(C,\omega_C(-D))\cong H^0(\bP^1,\cO_{\bP^1}(g-1-h_{\ast}D))
\end{align*}
where the last isomorphism is induced by $h:C\to \bP^1$, the hyperelliptic map determined by the unique $g^1_2$ and $h_{\ast}D$ is the degree $g-1-2r$ divisor defined in Notation \ref{notation: associated divisors on P1}.

For each $i<j$, let $\{L_k\}\subseteq W^i_{g-1}(C)^{\circ}$ be a sequence of line bundles converging to $L\in W^j_{g-1}(C)^{\circ}$. Using the isomorphism \eqref{eqn: the isomorphism between U and open part of BrillNoether}, we can write
\[ L_k=\cO_{C}(D_k)\otimes ig^1_2, \quad L=\cO_{C}(D)\otimes jg^1_2\]
such that $D_k\in U_{g-1-2i}$ and $D\in U_{g-1-2j}$. From the discussion above, we know that
\begin{align*}
    (N^{\ast}_{W^i_{g-1}(C)^{\circ}|\Jac(C)})_{L_k}=H^0(C,\omega_C(-D_k))\cong H^0(\bP^1,\cO_{\bP^1}(g-1-h_{\ast}D_k))\\
    (N^{\ast}_{W^j_{g-1}(C)^{\circ}|\Jac(C)})_{L}=H^0(C,\omega_C(-D))\cong H^0(\bP^1,\cO_{\bP^1}(g-1-h_{\ast}D))
\end{align*} 
If we denote $\overline{D}\colonequals \lim_{k\to \infty}D_k \in C_{g-1-2i}$ to be the limit divisor, since $\lim_{k\to \infty}L_k=L$, we see that $D$ is an effective subdivisor of $\overline{D}$. Therefore,
\begin{align*}
    \lim_{k\to \infty}(N^{\ast}_{W^i_{g-1}(C)^{\circ}|\Jac(C)})_{L_k}=&H^0(\bP^1,\cO_{\bP^1}(g-1-h_{\ast}\overline{D}))\\ 
    \subseteq &H^0(\bP^1,\cO_{\bP^1}(g-1-h_{\ast}D))=(N^{\ast}_{W^j_{g-1}(C)^{\circ}|\Jac(C)})_{L},
\end{align*}
where the first equality uses the fact that $H^1(\bP^1,\cO_{\bP^1}(k))=0$ for any $k\geq 0$ and hence we can take limits. This verifies the Whitney condition (A) for the pair
$(W^i_{g-1}(C)^{\circ},W^j_{g-1}(C)^{\circ})$, by going to the dual spaces. Now by
Lemma \ref{lemma: the strong condition W'}, in order to prove the Whitney condition
(B), we just need to show that there exists a constant $A$ such that 
\[ d(H^0(C,\omega_C(-\overline{D}),H^0(C,\omega_C(-D_k)))\leq A\cdot d(L,L_k)=A\cdot d(\overline{D},D_k),\]
where the distance function on the left is induced by an inner product on the vector
space $H^0(C,\omega_C)$ and the distance functions on the right are induced by the
Euclidean norm on a neighborhood of $\overline{D}$ in $C_{g-1-2i}$ and a neighborhood of $L$ in $\Jac(C)$, respectively. The last equality comes from the fact that $D_k\in U_{g-1-2r}$ and $\delta_{(g-1-2r)/2}$ is an isomorphism over $U_{g-1-2r}$. Since the
hyperelliptic map $h:C\to \bP^1$ is either a local isomorphism (off the branch locus) or
locally of the form $t \mapsto t^2$ (on the branch locus), we can push everything
down to $\bP^1$; there, it sufficies to prove that there exists $A$ so that
\[ d(H^0(\bP^1,\cO_{\bP^1}(g-1-h_{\ast}\overline{D})),H^0(\bP^1,\cO_{\bP^1}(g-1-h_{\ast}D_k)))\leq A\cdot d(h_{\ast}\overline{D},h_{\ast}D_k),\]
which follows from the interpretation of the space $H^0(\bP^1,\cO_{\bP^1}(g-1-h_{\ast}D))$ as the space of degree $g-1$ homogeneous polynomials vanishing along the divisor $h_{\ast}D$ and explicit computations.\footnote{Botong Wang pointed out that one can view this as a Lipschitz property of the map between compact manifolds $\Sym^{g-1-2i}\bP^1 \to \mathrm{Grass}(H^0(\bP^1,\cO_{\bP^1}(g-1)),2i)$ which sends $E$ to $H^0(\bP^1,\cO_{\bP^1}(g-1-E))$. }

For the pair $(\Jac(C)-\Theta, W^{i}_{g-1}(C)^{\circ})$, Condition (W) is vacuous
because $\Jac(C)$ is a complex manifold (using the property (1) of the distance
function in Fact \ref{fact about distance function}).
\end{proof}

\section{Questions and open problems}\label{sec: questions}

The log resolution of the hyperelliptic theta divisor is rather intricate. To have a better understanding of it, we ask
\begin{question}
Is there a modular interpretation of the log resolution in Theorem \ref{theorem: log resolution of hyperelliptic theta divisors}?
\end{question}

Let $C$ be a Brill-Noether general curve. The Brill-Noether varieties $W^{r}_{g-1}(C)$ behave like generic determinantal varieties. It is natural to ask for an extension of our results:

\begin{problem}
Prove that Theorem \ref{theorem: log resolution of hyperelliptic theta divisors} and Proposition \ref{prop: Brill Noether are Whitney} hold for such a curve $C$.\footnote{Note added Nov 2024: this problem is solved now by Budur in \cite{budur2024localstructurethetadivisors}.}
\end{problem}

\newpage
\appendix

\section{Reducedness of $W^r_d(C)$}

In this appendix, we provide the proof of Proposition \ref{prop: reducedness of Wrd}. This result is not needed for the paper because in Theorem \ref{theorem: log resolution of hyperelliptic theta divisors} we can always take $W^r_d(C)$ with its induced reduced structures. While this result is not essential, we include it here in response to a question posed by Nero Budur \cite{BudurDoan}.

Let $C$ be a smooth projective curve. Let $r,d$ be nonnegative integers. Set-theoretically, $W^r_d(C)$ is the space of line bundles $L\in \Pic^{d}(C)$ such that $h^0(L)\geq r+1$. Its scheme structure can be defined as follows, which is a variant of the method in \cite[Chapter IV,\S 3]{ACGH}. Let $L\in \Pic^{d}(C)$ be a closed point, as in the proof of Lemma \ref{lemma: reducedness of tangent cone} below, one can produce a vector bundle map $A:E^0\to E^1$ over $\Pic^d(C)$ where $E^0,E^1$ are vector bundles of rank $h^0(L)$ and $h^1(L)$ (see \eqref{eqn: definition of Wrd}). Then $W^r_d(C)$ is defined, in a neighborhood of $L$, as the $(h^0(L)-r)$-th determinantal variety associated to $A$; equivalently, it is cut out by all $(h^0(L)-r+1)\times (h^0(L)-r+1)$ minors of $A$. 

\begin{prop}\label{prop: reducedness of Wrd}
Let $C$ be a smooth hyperelliptic curve of genus $g$. Let $d,r\in \mathbb{N}$ be integers such that $0\leq r\leq d\leq g$. Then the Brill-Noether variety $W^r_d(C)$ is reduced.
\end{prop}

We recall the following result saying that reducedness can be checked on the level of tangent cones.
\begin{lemma}\label{lemma: reducedness via tangent cones}
Let $Z\subseteq X$ be a closed subscheme of a smooth variety $X$ and $x\in Z$ be a closed point. If the tangent cone $TC_xZ\subseteq T_xX$ is reduced, then $Z$ is reduced in an open neighborhood of $x$.
\end{lemma}

\begin{proof}
Equip the tangent cone $TC_xZ$ with its natural scheme structure, then there is a flat specialization of (a neighborhood of $x$ in)  $Z$ to $TC_xZ$.  The desired result follows from the fact that reducedness is an open condition in flat families, c.f. \cite[Theorem 12.1.1 (vii)]{EGA}.
\end{proof}

By Lemma \ref{lemma: reducedness via tangent cones}, Proposition \ref{prop: reducedness of Wrd} is reduced to the following
\begin{lemma}\label{lemma: reducedness of tangent cone}
The tangent cone $TC_LW^r_d(C)$ is reduced for any $L\in W^r_d(C)$.
\end{lemma}

\begin{proof}
To simplify the notation, we denote $W^r_d=W^r_d(C)$. Fix $\cL$ a Poincar\'e line bundle on $C\times \Pic^{d}(C)$ and let $\pr_2:C\times \Pic^{d}(C)\to \Pic^{d}(C)$ be the second projection. Let $L\in W^r_d$ be a line bundle of degree $d$ and  assume $h^0(L)=s+1$ with $s\geq r$. In a neighborhood of $L$ in $\Pic^d(C)$, we can produce a minimal complex computing $W^r_d$, by a variant of the method in \cite[Chapter IV, \S 3]{ACGH}. Note that we can always pick a point $p\in C$ such that $h^1(L(p))=h^1(L)-1$ and $H^0(L)=H^0(L(p))$. Iterating this, we can pick an effective divisor $D$ of degree $h^1(L)=g-d+s$ with the property that $H^1(L(D))=0$ and in the short exact sequence
\[ 0 \to L \to L(D) \to L(D)\otimes \cO_D \to 0,\]
the induced connecting map $H^0(D,L(D)\otimes \cO_D)\to H^1(C,L)$ is an isomorphism (equivalently, $H^0(C,L)\to H^0(C,L(D))$ is an isomorphism). Denote by $\cD=\pr_2^{\ast}D$ the effective divisor on $C\times \Pic^d(C)$. Then on some neighborhood of the point $L$, we have a short exact sequence
\[ 0 \to \pr_{2,\ast}\cL(\cD) \to \pr_{2,\ast}(\cL(\cD)\otimes \cO_{\cD}) \to R^1\pr_{2,\ast}\cL \to 0,\]
where $R^1\pr_{2,\ast}\cL(\cD)$ vanishes on the neighborhood in question. Here $\pr_{2,\ast}\cL=0$ because it is torsion-free and vanishes at a general point in the neighborhood of $L$. This gives us a presentation
\begin{equation}\label{eqn: definition of Wrd} 0 \to E^0 \xrightarrow{A} E^1 \to R^1\pr_{2,\ast}\cL \to 0\end{equation}
where $E^0$ and $E^1$ are vector bundles of rank $h^0(L)=s+1$, respectively $h^1(L)=g-d+s$. Moreover, the differential, viewed as a matrix $A$, vanishes at the point $L$. Let $A_1$ be the linear part of $A$; that has entries in $\fm/\fm^2$, where $\fm$ is the maximal ideal at $L$.

Now $W^r_d$ is defined, in a neighborhood of the point $L$, by the vanishing of all the $(s-r+1)\times (s-r+1)$ minors of $A$. Because for $L'\in W^r_d$, the condition is
\begin{align*}
h^0(L')\geq r+1 &\Leftrightarrow h^1(L')\geq g-d+r \\
&\Leftrightarrow\mathrm{rank}(A)_{L'}\leq (g-d+s)-(g-d+r)=(s-r).
\end{align*}
It follows from the tangent cone theorem in generic vanishing theory (c.f. \cite[Theorem 4]{GL87}) that one has the following containments:
\[ \cI_{1} \subseteq \cI_{TC_L W^r_d} \subseteq \sqrt{\cI_{TC_L W^r_d}},\]
where the first ideal is generated by all the $(s-r+1)\times (s-r+1)$ minors of $A_1$. If one knows that the first ideal  $ \cI_{1} $ is a radical ideal, and that both $\cI_1$ and  $\sqrt{\cI_{TC_L W^r_d}}$ define the same conical subset in $T_L\Pic^d(C)$ (forgetting about the scheme structure), then the tangent cone $TC_L W^r_d$ is reduced. 

Since $W^d_r\cong W_{d-2r}$ as sets, one has 
\[ \dim TC_LW^r_d=\dim W^r_d=d-2r.\] 
By the discussion above, it suffices to show that the $(s-r+1)\times (s-r+1)$ minors of the matrix $A_1$ defines a reduced, irreducible subscheme of $T_L\Pic^d(C)$ of dimension $d-2r$. This boils down to the following two claims.

\textbf{Claim 1}: The matrix $A_1$ is a Hankel/Catalecticant matrix,  i.e.
\[ A_1 = \left(\begin{matrix} x_1 & x_2 & \cdots & x_{g-d+s} \\
x_2 & x_3 & \cdots &x_{g-d+s+1} \\
\cdots  & \cdots & \cdots & \cdots\\
x_{s+1} & \cdots & \cdots & x_{g-d+2s}
\end{matrix}\right)\]
up to a change of local coordinates.

\textit{Proof of Claim 1}: We learned this argument from Nero Budur, see \cite[Proposition 8.17]{BudurDoan}. By \cite{ACGH}, the matrix $A_1$ is the one given by the map
\[ H^0(L) \to H^1(L)\otimes H^0(\omega_C),\]
which is equivalent to the Petri map
\[ \pi_L:H^0(L)\otimes H^0(\omega_C\otimes L^{-1}) \to H^0(\omega_C).\]
Since $C$ is hyperelliptic and $h^0(L)=s+1$, we can write 
\begin{align*}
L&=sg^1_2+p_1+\cdots+p_{d-2s}\\
\omega_C\otimes L^{-1}&=(g-1-s-(d-2s))g^1_2+q_1+\cdots+q_{d-2s},
\end{align*}
where $p_i+q_i$ is a hyperelliptic pair for each $1\leq i\leq d-2s$ and no two $p_i$ lie in the same fiber of the hyperelliptic involution $C\to \bP^1$. Then the Petri map corresponds to
\[ H^0(\bP^1,\cO(s))\otimes H^0(\bP^1,\cO(g-1-d+s))\to H^0(\bP^1,\cO(g-1-d+2s)) \to H^0(\bP^1,\cO(g-1))\]
The last map is the tensor product with the section $\eta\in H^0(\bP^1,\cO(d-2s))$, where $\eta$ is the product of all linear forms defining the image of $p_i$ in $\bP^1$ for $1\leq i\leq d-2s$. Write $V=H^0(\bP^1,\cO(1))$, then the Petri map is the natural multiplication map
\[ \Sym^sV \otimes \Sym^{g-1-d+s}V \to \Sym^{g-1-d+2s}V,\]
which clearly gives a Catalecticant matrix since $\dim V=2$.

\textbf{Claim 2}: Let $C_{v,w}$ be a $v\times w$ Catalecticant matrix with $v\geq w$, i.e.
\[ C_{v,w} = \left(\begin{matrix} x_1 & x_2 & \cdots & x_{w} \\
x_2 & x_3 & \cdots &x_{w+1} \\
\cdots  & \cdots & \cdots & \cdots\\
x_{v} & \cdots & \cdots & x_{v+w-1}
\end{matrix}\right)\]
Then for $k<w$, the ideals of $(k+1)\times (k+1)$ minors of $C_{v,w}$ defines a reduced irreducible subscheme $Z$ of dimension $2k$ in $\C^{v+w-1}$.

\textit{Proof of Claim 2}: We use notations in \cite{Eisenbud88}. Let $M=\mathrm{Cat}(v,w)\subseteq \bP\C^{vw}$ be the Catalecticant space, which is of dimension $v+w-2$ (c.f. \cite[Page 561]{Eisenbud88}). Let $M_k$ be the subscheme of matrices of rank $\leq k$ in $M$, the linear space corresponds to all the minors of $C_{v,w}$ of size $k+1$.  By \cite[Proposition 4.3]{Eisenbud88}, one has
\[ \mathrm{codim}_M M_k=v+w-1-2k,\]
and $M_k$ is the $k$-secant variety of a rational normal curve. Thus 
\[\dim M_k=\dim M-(v+w-1-2k)=(v+w-2)-(v+w-1-2k)=2k-1\]
and $M_k$ is irreducible. Moreover, Eisenbud \cite[Proposition 4.3 and after]{Eisenbud88}
observes that $M_k$ is reduced. Therefore the corresponding space $Z\subseteq \C^{v+w-1}$ is reduced, irreducible and has dimension $2k$.

Now we can finish the proof of this lemma. If $d\geq g-1$, then $s-r<s+1\leq g-d+s$; if $d=g$, then we can assume $r\geq 1$ ($W^0_g(C)$ is reduced by a theorem of Kempf) and still get $s-r< s=g-d+s\leq s+1$. Therefore we can apply Claim 1 and 2 to obtain that the $(s-r+1)\times (s-r+1)$ minors of the matrix $A_1$ defines a reduced, irreducible subscheme in $T_L\Pic^d(C)=\C^g$ of dimension $2(s-r)+(d-2s)$ (because only the variables $x_1,\cdots,x_{g-d+2s}$ show up in the matrix $A_1$ and the other variables provide an additional $d-2s$ dimensions). This gives what we want and therefore we finish the proof that $TC_LW^r_d$ is reduced. 

As a consequence, $W^r_d(C)$ is reduced for any $0\leq r\leq d\leq g$.

\end{proof}

\bibliographystyle{abbrv}
\bibliography{main}{}

\begin{thebibliography}{10}

\bibitem{ACGH}
E.~Arbarello, M.~Cornalba, P.~A. Griffiths, and J.~Harris.
\newblock {\em Geometry of algebraic curves. {V}ol. {I}}, volume 267 of {\em
  Grundlehren der mathematischen Wissenschaften [Fundamental Principles of
  Mathematical Sciences]}.
\newblock Springer-Verlag, New York, 1985.

\bibitem{Bertram92}
A.~Bertram.
\newblock Moduli of rank-{$2$} vector bundles, theta divisors, and the geometry
  of curves in projective space.
\newblock {\em J. Differential Geom.}, 35(2):429--469, 1992.

\bibitem{budur2024localstructurethetadivisors}
N.~Budur.
\newblock Local structure of theta divisors and related loci of generic curves.
\newblock {\em arXiv preprint arXiv: 2311.09250}, 2024.

\bibitem{BudurDoan}
N.~Budur and A.-K. Doan.
\newblock Deformations with cohomology constraints: a review.
\newblock {\em arXiv preprint arXiv: 2311.08052}, 2023.

\bibitem{Ein91}
L.~Ein.
\newblock Normal sheaves of linear systems on curves.
\newblock In {\em Algebraic geometry: {S}undance 1988}, volume 116 of {\em
  Contemp. Math.}, pages 9--18. Amer. Math. Soc., Providence, RI, 1991.

\bibitem{Eisenbud88}
D.~Eisenbud.
\newblock Linear sections of determinantal varieties.
\newblock {\em Amer. J. Math.}, 110(3):541--575, 1988.

\bibitem{GL87}
M.~Green and R.~Lazarsfeld.
\newblock Deformation theory, generic vanishing theorems, and some conjectures
  of {E}nriques, {C}atanese and {B}eauville.
\newblock {\em Invent. Math.}, 90(2):389--407, 1987.

\bibitem{EGA}
A.~Grothendieck.
\newblock \'{E}l\'{e}ments de g\'{e}om\'{e}trie alg\'{e}brique. {IV}. \'{E}tude
  locale des sch\'{e}mas et des morphismes de sch\'{e}mas. {III}.
\newblock {\em Inst. Hautes \'{E}tudes Sci. Publ. Math.}, (28):255, 1966.

\bibitem{Johnson}
A.~A. Johnson.
\newblock {\em Multiplier ideals of determinantal ideals}.
\newblock ProQuest LLC, Ann Arbor, MI, 2003.
\newblock Thesis (Ph.D.)--University of Michigan.

\bibitem{Kuo}
T.-C. Kuo.
\newblock The ratio test for analytic {W}hitney stratifications.
\newblock In {\em Proceedings of {L}iverpool {S}ingularities---{S}ymposium, {I}
  (1969/70)}, Lecture Notes in Mathematics, Vol. 192, pages 141--149. Springer,
  Berlin, 1971.

\bibitem{SY23}
C.~Schnell and R.~Yang.
\newblock Higher multiplier ideals.
\newblock {\em arXiv preprint arXiv:2309.16763}, 2023.

\bibitem{Vainsencher}
I.~Vainsencher.
\newblock Complete collineations and blowing up determinantal ideals.
\newblock {\em Math. Ann.}, 267(3):417--432, 1984.

\bibitem{Verdier76}
J.-L. Verdier.
\newblock Stratifications de {W}hitney et th\'{e}or\`eme de {B}ertini-{S}ard.
\newblock {\em Invent. Math.}, 36:295--312, 1976.

\end{thebibliography}

\vspace{2cm}

\footnotesize{
\textsc{Department of Mathematics, Stony Brook University, Stony Brook, New York 11794, United States} \\
\indent \textit{E-mail address:} \href{mailto:cschnell@math.stonybrook.edu}{cschnell@math.stonybrook.edu}

\vspace{\baselineskip}

\textsc{Department of Mathematics, University of Kansas, Lawrence, Kansas 66045, United States} \\
\indent \textit{E-mail address:} \href{mailto:ruijie.yang@hu-berlin.de}{ruijie.yang@ku.edu}
}

\end{document}